\documentclass[11pt,twoside]{amsart}
\usepackage{mathrsfs}
\usepackage{txfonts}
\usepackage{bbm}
\usepackage{amsfonts}
\usepackage{amsmath}
\usepackage{esint}
\usepackage{amssymb}
\usepackage{wasysym}
\usepackage{psfrag}
\usepackage{graphics,epsfig,amsmath}
\usepackage{comment}
\usepackage{enumerate}
\usepackage{fancyhdr}

\newtheorem{Theorem}{Theorem}[section]
\newtheorem{Lemma}[Theorem]{Lemma}

\newtheorem{Definition}{\indent Definition}[section]
\newtheorem{remark}[Theorem]{Remark}

\usepackage{metalogo}
\usepackage{color}
\selectfont
\usepackage{mathtools}
\mathtoolsset{showonlyrefs}
\title[Regularity for mixed operators]{Regularity results for solutions \\ of mixed local and nonlocal elliptic equations}
\author[X. Su]{Xifeng Su}
\address{School of Mathematical Sciences, Laboratory of Mathematics and Complex Systems (Ministry of Education)\\
Beijing Normal University,
No. 19, XinJieKouWai St., HaiDian District, Beijing 100875, P. R. China}
\email{xfsu@bnu.edu.cn, billy3492@gmail.com}

\author[E. Valdinoci]{Enrico Valdinoci}
\address{Department of Mathematics and Statistics, University of Western Australia, 35 Stirling Highway, WA 6009 Crawley, Australia}
\email{enrico.valdinoci@uwa.edu.au}
	
\author[Y. Wei]{Yuanhong Wei}
\address{School of Mathematics, Jilin University, No. 2699, Qianjin St., Changchun 130012, P. R.  China}
\email{weiyuanhong@jlu.edu.cn}

\author[J. Zhang]{Jiwen Zhang}
\address{School of Mathematical Sciences, Beijing Normal University, No. 19, XinJieKouWai St., HaiDian District, Beijing 100875, P. R. China}
\email{jwzhang628@mail.bnu.edu.cn}


\thanks{Both X. Su and Y. Wei are supported by the National Natural Science Foundation of China (Grant No. 11971060, 11871242). Y. Wei is supported by Natural Science Foundation of Jilin Province (Grant No. 20200201248JC), and Scientific Research Project of Education Department of Jilin Province (Grant No. JJKH20220964KJ)}

\begin{document}
\maketitle

\begin{abstract}
We consider the mixed local-nonlocal semi-linear elliptic equations driven by the superposition of Brownian and L\'evy processes
\begin{equation*}
\left\{
\begin{array}{ll}
    - \Delta u +  (-\Delta)^s u = g(x,u) & \hbox{in $\Omega$,} \\
    u=0 &  \hbox{in $\mathbb{R}^n\backslash\Omega$.} \\
\end{array}
\right.
\end{equation*}
Under mild assumptions on the nonlinear term $g$, we show the $L^\infty$ boundedness of  any weak solution (either not changing sign or sign-changing) by the Moser iteration method. Moreover, when $s\in (0, \frac{1}{2}]$, we obtain that the solution is unique and actually belongs to $C^{1,\alpha}(\overline{\Omega})$ for any $\alpha\in (0,1)$.
\end{abstract}

\section{Introduction}
The present paper is concerned with the regularity results of elliptic equations driven by a special subclass of mixed differential and pseudo-differential elliptic operators
\begin{equation}\label{eq:operator L}
\mathcal{L}=-\Delta  + (-\Delta)^s, \qquad \text{ for some  } s\in(0,1).
\end{equation}
Here, $ (-\Delta )^{s} $ is defined as
\begin{equation}
(-\Delta )^{s} u(x)=c_{n,s}\,{\text PV}\int_{\mathbb{R}^{n}} \frac{u(x)-u(y)}{|x-y|^{n+2s}}\, dy,
\end{equation}
 where $ c_{n,s} $ is a suitable normalization constant, whose explicit value only plays a minor role in this paper,
 and~${\text PV}$ means that the integral is taken in the Cauchy Principal Value sense.

The operator $\mathcal{L}$ naturally arises as the superposition of a classical random walk and a L\'evy flight. For instance, as observed in~\cite{DV21}, these operators describe a biological species whose individuals diffuse either by a random walk or by a jump process, according to prescribed probabilities.
The analysis of different types of mixed operators motivated by biological questions
has also carried out in~\cite{MR3082317, MR3771424, Dipierro2021POU}.

Moreover, mixed operators have recently received a great attention from
different points of view, including regularity theory~\cite{MR2129093, MR2422079, MR2653895, MR2911421, MR2912450, MR3724879, MR4381148, MR4387204, BIMUVE}, existence and non-existence results~\cite{MR4275496, 67SALORT, FBENIDAND12}, eigenvalue problems~\cite{MR4044581, MR4400914, Dipierro2021LinearTF},
shape optimization and calculus of variations~\cite{MR4391102, FABEJK102},
symmetry and rigidity results~\cite{MR4313576}, etc.

One interesting challenging aspect of this topic is that
it combines the classical setting and the
features typical of nonlocal operators in a framework that is not scale-invariant. Hence,
at different scales, different features of the classical/nonlocal world tend to prevail, or they end up coexisting
into new interesting phenomena.

The goal of this paper is to show the $L^\infty$ and $C^{1, \alpha}$ regularity for the weak solutions of
 \begin{equation}\label{eq:g main equation}
 	\begin{cases}
 	\mathcal{L} u= g(x,u) \quad &\text{in } \Omega, \\
 	u=0 & \text{in } \mathbb{R}^{n}\backslash \Omega,
 	\end{cases}
 	\end{equation}
	where a suitable notion of weak solutions associated to \eqref{eq:g main equation} on the space $X_{0}^1$  will be detailed in Definition~\ref{definition of weak solution} of Section~\ref{sec:preliminary}.
	
In order to get right to the point, we will suppose in the present paper that the existence of weak solutions of \eqref{eq:g main equation}   is known. As a matter of fact, weak solutions are often obtained via variational methods and nonlinear analysis tools, e.g. one may refer \cite{SVWZ} for the existence of weak solutions (both not changing sign and sign-changing)  under standard nonlinear analysis assumptions.

We  assume once and for all that~$ s\in (0,1)$ is given, $ n>2s,$ and $ \Omega \subset \mathbb{R}^{n} $ is an open bounded set  satisfying some smooth boundary conditions, which will be specified as needed.

We  will first derive an $L^\infty$ regularity result for any weak solution of the elliptic problem~\eqref{eq:g main equation}. More precisely, denoting~$2^*_s:=\frac{2n}{n-2s}$, we have:

\begin{Theorem}\label{th: regularity}
Let $ u\in X_{0}^1 $  be   a weak solution of \eqref{eq:g main equation} and $ s\in(0,1) $.
	Assume that there exist $ c>0 $ and $q\in [2,2^*_s]$ such that
\begin{equation}\label{eq: H 1}
 		|g(x,t)|\leqslant c(1+|t|^{q-1}) \quad \text{ for a.e. } x\in  \Omega ,  t\in \mathbb{R}.
\end{equation}Then, $ u\in L^{\infty}(\Omega)$. More precisely, there exists a constant $C_0= C_0(c,n, s, \Omega)>0$ independent of $u$, such that
	\begin{equation*}
		\|u\|_{\infty}\leqslant C_0\left(1+\int_{\Omega}|u|^{2_s^*\beta_1}\,dx\right)^{\frac{1}{2_s^*(\beta_1-1)}},
	\end{equation*}where~$\beta_1:=\frac{2_s^*+1}{2}$ and $\|\cdot\|_\infty := \|\cdot\|_{L^\infty(\Omega)}$.
	\end{Theorem}

We remark that the weak solutions here could be both not changing sign and sign-changing.  In the case of the fractional Laplace\footnote{As a technical remark, the additional presence of the Laplace operator will produce an extra term (that is the first term in \eqref{integration of the main equation}). We will check that this term is nonnegative by using a suitable integration by parts for  the second-order generalized derivative, which will allow us to successfully complete the estimate produced by  the full mixed order operator.}operator $(-\Delta )^s$,
the $L^\infty$ boundedness is also proved respectively in \cite{BCSS15,DMVbook} for a positive weak solution and \cite{WS15, WS18} for a weak solution which could change sign.
\medskip

As is standard in the literature (see e.g. \cite{bensoussan1984impulse, gimbert1984existence}),
the  $L^\infty$ regularity of weak solutions allows one to obtain a global $ C^{1,\alpha} $-regularity theory, which relies on   the $ W^{2,p} $-regularity theory.

We will pursue this direction and  our $ C^{1,\alpha} $-regularity theorem of the mixed elliptic equation \eqref{eq:g main equation} with $ s\in\left(0,\frac{1}{2}\right] $ goes as follows.
\begin{Theorem}\label{theorem C^1 alpha}
	 Let us assume, in addition to the hypothesis of Theorem~\ref{th: regularity}, that $ \partial\Omega $ is of class  $ C^{1,1} $ and $ s\in \left(0,\frac{1}{2}\right]$, with $n>2s $. Then,
	\begin{equation}\label{eq:regularity C^1,alpha}
		u\in C^{1,\alpha}(\overline\Omega) \qquad \text{ for any   } \alpha \in (0,1).
	\end{equation}
\end{Theorem}

\begin{remark}\label{REMARK for some beta}
{\rm As observed e.g. in \cite[Theorem 2.7]{FABEJK102} and in the references therein, H\"older estimates
for the gradient of the solution remain valid for all~$s\in(0,1)$, in the sense that
one can prove in such a generality that~$u\in C^{1,\beta}(\overline\Omega)$ for some~$ \beta \in (0,1)$
(for instance, such a result would follow by combining Theorem~\ref{theorem C^1 alpha} and Lemma~\ref{Lp estimate for some p}).
The interest of Theorem~\ref{theorem C^1 alpha} is in finding a precise H\"older exponent for this type of regularity theory
in the range~$s\in\left(0,\frac{1}{2}\right]$.}
\end{remark}

Note that Theorem~\ref{theorem C^1 alpha}  is just an immediate corollary of  the following  $W^{2,p} $-regularity theorem.
\begin{Theorem}\label{theorem w 2p}
	Let $ \Omega $ be $ C^{1,1} $ domain in $ \mathbb{R}^n $ and $ s\in \left(0,\frac{1}{2}\right],n>2s$. Then if $ f\in L^p(\Omega) $ with $ 1<p<+\infty $, the problem
	\begin{equation}\label{eq: main equation}
	-\Delta u+(-\Delta)^su=f, \qquad \text{in } \Omega
	\end{equation}
	has a unique solution $ u\in W^{2,p}(\Omega)\cap W^{1,p}_0(\Omega)$.
	Furthermore,
	\begin{equation*}
	\|u\|_{W^{2,p}(\Omega)}\leqslant C_1\bigg(\|u\|_{p}+\|f\|_{p}\bigg),
	\end{equation*}
	where the constant $C_1$ depends on $ \Omega, n, s, p $ and $\|\cdot\|_p:=\|\cdot\|_{L^p(\Omega)}$ for short.
\end{Theorem}

 \section{Some preliminary facts}\setcounter{equation}{0}\label{sec:preliminary}
In this section, we provide several definitions and basic facts on the weak solutions of the Dirichlet problem associated with the mixed operator $ \mathcal{L} $ in \eqref{eq:g main equation}, that is
\begin{equation*}
\begin{cases}
\mathcal{L} u= g(x,u) \quad &\text{in } \Omega, \\
u=0 & \text{in } \mathbb{R}^{n}\backslash \Omega.
\end{cases}
\end{equation*}
See also~\cite{SVWZ} for the existence and multiplicity results of weak solutions.

Let $ s\in(0,1) $ be given and $ \Omega\subset \mathbb{R}^{n}$ be an open bounded set with $ C^{1} $ boundary where $n>2s$.
A ``natural'' space to consider is the following (see e.g.~\cite{MR4387204}):
\begin{equation*}
X^1:=\left\{ u:\mathbb{R}^{n}\rightarrow \mathbb{R} \text{ is Lebesgue measurable: } u|_{\Omega} \in H^{1}(\Omega);\frac{|u(x)-u(y)|}{|x-y|^{\frac{n+2s}{2}}}\in L^{2}(\mathcal{\mathbb{R}}^{2n}) \right\}.
\end{equation*}
The norm of $u\in X^1 $ is defined as follows:
\begin{equation*}
\|u\|_{X^1}=\bigg(\|u\|_{H^{1}(\Omega)}^{2}+\int\!\!\!\int\nolimits_{\mathbb{R}^{2n}}\frac{|u(x)-u(y)|^2}{|x-y|^{n+2s}}\, dxdy\bigg)^{\frac{1}{2}} .
\end{equation*}
Our working space for the weak solutions would be:
\begin{equation}
X_{0}^1\equiv \left\{u\in X^1: u=0 \text{ a.e. in  } \mathbb{R}^n\backslash \Omega \right\}.\label{workingspace}
\end{equation}
\begin{remark}\label{X01impliesH01}{\rm
	Since $ \Omega $ has $ C^{1} $ boundary, any function $ u\in X_{0}^1 $ satisfies
	\begin{equation*}
	u|_{\Omega} \in H^{1}_{0}(\Omega).
	\end{equation*}}
\end{remark}
Due to Remark~\ref{X01impliesH01}, the norm in $ X_{0}^1 $ is also equivalent to
\begin{equation}
\|u\|_{X_{0}^1} :=\bigg(\|\nabla u\|^{2}_{2}+\int\!\!\!\int\nolimits_{\mathbb{R}^{2n}}\frac{|u(x)-u(y)|^2}{|x-y|^{n+2s}}\, dxdy\bigg)^{\frac{1}{2}}.  \label{Xnorm}
\end{equation}
Obviously, $X^1_{0} $ is a Hilbert space equipped with the inner product
\begin{equation*}
\left\langle u,v \right\rangle_{X_{0}^1}=\int_{\Omega} \nabla u\cdot  \nabla v\, dx+\int\!\!\!\int_{\mathbb{R}^{2n}}\frac{(u(x)-u(y))(v(x)-v(y))}{|x-y|^{n+2s}}\, dxdy, \quad \forall u,v \in X_{0}^1.
\end{equation*}
\begin{Definition}[Weak solution]\label{definition of weak solution}
	We say that $ u\in X_{0}^1 $ is a weak solution of the mixed elliptic equation \eqref{eq:g main equation}, if $ u $ satisfies
\begin{equation}
	\displaystyle\int_{\Omega}\left\langle \nabla u,\nabla \varphi \right\rangle + \displaystyle\int\!\!\!\int\nolimits_{\mathbb{R}^{2n}}\frac{(u(x)-u(y))(\varphi(x)-\varphi(y))}{|x-y|^{n+2s}}\, dxdy
	=\int_{\Omega}g(x,u(x))\varphi(x)\, dx,      \label{weaksolution}
	\end{equation}
	for any $ \varphi \in X_{0}^1 $.
	\end{Definition}

To begin with, we investigate some key facts of $ X_0^1 $ in the following embedding lemma.
\begin{Lemma}\label{lemma: embedding continuous}
	The embedding $ X_{0}^1\hookrightarrow L^{2^*_s}(\Omega) $ is continuous where $ 2^*_s=\frac{2n}{n-2s}. $
\end{Lemma}
\begin{proof}
	Thanks to \cite{DPV}, we know for all $ v\in X_0^1 $, $ v\in H^s(\mathbb{R}^n), $ and
	\begin{equation*}
	\|v\|^2_{L^{2^*_s}(\mathbb{R}^n)}\leqslant S\int\!\!\!\int\nolimits_{\mathbb{R}^{2n}}\frac{|v(x)-v(y)|^2}{|x-y|^{n+2s}}\, dxdy,
	\end{equation*}
	where $ S $ is a positive constant depending on $ n $ and $ s $. It follows that
	\begin{equation*}
	\|v\|^2_{2^*_s}=\|v\|^2_{L^{2^*_s}(\mathbb{R}^n)}\leqslant S\int\!\!\!\int\nolimits_{\mathbb{R}^{2n}}\frac{|v(x)-v(y)|^2}{|x-y|^{n+2s}}\, dxdy\leqslant S\|v\|_{X_0^1}^2.
	\end{equation*}
	Thus, the embedding $ X_{0}^1\hookrightarrow L^{2^*_s}(\Omega) $ is continuous.
\end{proof}

Let $$\mathscr{C}_0^\infty:=\left\{u\in C(\mathbb{R}^n): u=0 \text{ in }\mathbb{R}^n\backslash \Omega, \;u|_{\Omega}\in C_0^\infty(\Omega) \right\} $$ and denote by $\overline{\mathscr{C}_0^\infty}$ the closure of $ \mathscr{C}_0^\infty $ with respect to $X_0^1$-norm. Then we have
\begin{Lemma}\label{lemma densely embedding}
	 $ \overline{\mathscr{C}_0^\infty} = X_0^1 $.
\end{Lemma}
\begin{proof}
	(i). We claim that
	\begin{equation}\label{We70P} \mathscr{C}_0^\infty\subset X^1_0 .\end{equation} Note that this and the fact that $ X_0^1 $ is complete would imply that $ \overline{\mathscr{C}_0^\infty}\subset X_0^1 $.
	
To check~\eqref{We70P} we proceed as follows. For any $ u\in \mathscr{C}_0^\infty $, there exists a compact subset $ K\subset\Omega $ such that $ u\equiv 0 $ in $ \mathbb{R}^n\backslash K$. We split the square of the norm $\|u\|_{X_0^1}$ into three parts
	\begin{equation} \label{eq:decomposition of norm}
	\begin{split}
	\|u\|_{X_0^1}^2&=\|u\|^2_{H^1_0(\Omega)}+	\int\!\!\!\int\nolimits_{\mathbb{R}^{2n}}\frac{|u(x)-u(y)|^2}{|x-y|^{n+2s}}\, dxdy \\
	&=\|u\|^2_{H^1_0(\Omega)}+\int\!\!\!\int_{\Omega\times\Omega}\frac{|u(x)-u(y)|^2}{|x-y|^{n+2s}}\, dxdy+2\int_{\Omega}\int_{\mathbb{R}^{n}\backslash\Omega}\frac{|u(x)|^2}{|x-y|^{n+2s}}\, dydx\\
	&=\uppercase\expandafter{\romannumeral 1}+\uppercase\expandafter{\romannumeral 2}+\uppercase\expandafter{\romannumeral 3}.
	\end{split}
	\end{equation}
{F}rom \cite[Proposition~2.2]{DPV}, we obtain\footnote{As a notation remark,
the convention used here is that
$$ \int_{A}\int_B f(x,y)\, dydx:=
\int_{A}\left(\int_B f(x,y)\, dy\right)dx.$$} that
	\begin{equation}\label{eq:Omega times Omega}
	\uppercase\expandafter{\romannumeral 2}\leqslant c_0\|u\|^2_{H^1_0(\Omega)}, \qquad \text{where $c_0=c_0(n,s) $ is a constant.}
	\end{equation}
	 It remains to show that $\uppercase\expandafter{\romannumeral 3}$ is bounded. For any $ y\in \mathbb{R}^n\backslash \Omega$, we have that
	\begin{equation*}
		\frac{|u(x)|^2}{|x-y|^{n+2s}}\leqslant\chi_K(x)|u(x)|^2\sup\limits_{x\in K}\frac{1}{|x-y|^{n+2s}},
	\end{equation*}
	and so
	\begin{equation}\label{eq: K outside}
		\int_{\Omega}\int_{\mathbb{R}^{n}\backslash\Omega}\frac{|u(x)|^2}{|x-y|^{n+2s}}\, dydx\leqslant \int_{\mathbb{R}^{n}\backslash\Omega}\frac{1}{\text{  dist}(y,\partial K)^{n+2s}}\, dy\,\|u\|^2_2.
	\end{equation}
We stress that the integral in \eqref{eq: K outside} is finite since dist$ (\partial \Omega,\partial K)\geqslant\alpha>0$. Combining \eqref{eq:decomposition of norm} and \eqref{eq: K outside}, we conclude that~$ u\in X_0^1$, thus proving~\eqref{We70P} as desired.
	
	(ii). On the other hand,
	since $ C_0^\infty(\Omega) $ is dense in $ H_0^1(\Omega) $, for all~$  u\in X_0^1\subset H_0^1(\Omega) $, there exists  $ \{u_m\}\subset C_0^\infty(\Omega) $ such that
	\begin{equation}\label{eq:C dense H}
	\|u_m-u\|_{H_0^1(\Omega)}\rightarrow 0 \text{ and } u_m\rightarrow u \text{ a.e. in } \Omega,\quad m\rightarrow +\infty.
	\end{equation}
	We define $ u_m\equiv 0 $ in $ \mathbb{R}^n\backslash \Omega $ for every $m\in \mathbb{N}$. Then $ \{u_m\}\subset \mathscr{C}_0^\infty $.
	
	We claim  that \begin{equation}\label{umisCA}{\mbox{$ \{u_m\} $ is a Cauchy sequence in $ X_0^1 $. 	}}\end{equation}
	Indeed, notice that half of $\uppercase\expandafter{\romannumeral 3}$ in \eqref{eq:decomposition of norm}  could be divided into the following two integrals and be estimated respectively by
	\begin{equation}\label{eq: >1 of third term }
	\begin{split}
	&\int_{\Omega}\int_{(\mathbb{R}^{n}\backslash\Omega)\cap\{|x-y|\geqslant 1\}}\frac{|(u_j-u_k)(x)|^2}{|x-y|^{n+2s}}\, dydx\\
	\leqslant&	\int_{\Omega}\bigg(\int_{\{|z|\geqslant 1\}}\frac{1}{|z|^{n+2s}}\, dz\bigg)|(u_j-u_k)(x)|^2dx<c_1\|(u_j-u_k)\|_{H_0^1(\Omega)}^2,
	\end{split}
	\end{equation}
	and
	\begin{equation}
	\begin{split}
	&\int_{\Omega}\int_{(\mathbb{R}^{n}\backslash\Omega)\cap\{|x-y|< 1\}}\frac{|(u_j-u_k)(x)|^2}{|x-y|^{n+2s}}\, dydx\\
	\leqslant& \int_{\Omega}\int_{(\mathbb{R}^{n}\backslash\Omega)\cap\{|z|< 1\}}\frac{\int_{0}^{1}|\nabla (u_j-u_k)(x+tz)|^2\, dt}{|z|^{n+2s-2}}\, dzdx\\
	\leqslant& \int_{\mathbb{R}^n}\int_{(\mathbb{R}^{n}\backslash\Omega)\cap\{|z|< 1\}}\frac{\int_{0}^{1}|\nabla (u_j-u_k)(x+tz)|^2\, dt}{|z|^{n+2s-2}}\, dzdx\\
	\leqslant&\ \ \| u_j-u_k \|^2_{H^1(\mathbb{R}^n)}\int_{\{|z|< 1\}}\frac{1}{|z|^{n+2s-2}}\, dz
	\leqslant c_2\|u_j-u_k\|^2_{H^1_0(\Omega)},
	\end{split}\label{eq:<1 of third term cauchy }
	\end{equation}
	where  $ c_1,c_2>0 $ are  constants independent of $\{u_m\}$.
	Due to \eqref{eq:decomposition of norm}, \eqref{eq:Omega times Omega}
	and~\eqref{eq: K outside}, we thereby obtain~\eqref{umisCA}
	as desired.
This yields that $X_0^1\subset  \overline{\mathscr{C}_0^\infty}$. 	
\end{proof}

 \begin{remark}{\rm
		Thanks to \eqref{eq:Omega times Omega}, \eqref{eq: >1 of third term } and \eqref{eq:<1 of third term cauchy }, we see that, for every $ u\in X_0^1, $
		\begin{equation*}
		\int\!\!\!\int\nolimits_{\mathbb{R}^{2n}}\frac{|u(x)-u(y)|^2}{|x-y|^{n+2s}}\, dxdy\leqslant C\|u\|_{H^1_0(\Omega)},
		\end{equation*}
		which implies that the norm~$ \|\cdot\|_{X_0^1}$ is equivalent to~$\|\cdot\|_{H_0^1(\Omega)} $ in the space $X_0^1$.}
\end{remark}

\section{$L^{\infty}$ regularity for any weak solution}\label{sec:regularity}\setcounter{equation}{0}

In this section, we show that any weak solution $u\in X_0^1$ of \eqref{eq:g main equation} is actually of 
 $L^{\infty}(\Omega)$ class. To prove this result, we will use the Moser iteration method.
\begin{proof}[ Proof of Theorem \ref{th: regularity}]
	We will divide the proof into the following three steps.
	
	\par\noindent \textbf{Step 1}. We construct an auxiliary function $\varphi$ and provide several fundamental properties of $\varphi$, which are useful for the iterative procedure.
	\vspace{2mm}
	
	Given $ \beta> 1 $ and $ T>0 $, we define
	\begin{equation*}
	\varphi(t)=
	\begin{cases}
	-\beta T^{\beta-1}(t+T)+T^\beta, &\text{  if }t\leqslant -T\\
	|t|^\beta, &\text{  if } -T<t<T,\\
	\beta T^{\beta-1}(t-T)+T^\beta, &\text{  if }t\geqslant T.
	\end{cases}
	\end{equation*}
	From the definition of $ \varphi(t) $, we may compute $ \varphi' $ and $ \varphi ''$ (in the sense of distributions), that is
	\begin{equation*}
	\varphi'(t)=
	\begin{cases}
	-\beta T^{\beta-1}, &\text{  if }t< -T\\
	-\beta(-t)^{\beta-1}, &\text{  if } -T\leqslant t\leqslant 0,\\
	\beta t^{\beta-1}, &\text{  if } 0\leqslant t\leqslant T,\\
	\beta T^{\beta-1}, &\text{  if } t>T,
	\end{cases}
	\end{equation*}
	and
	\begin{equation*}
	\varphi''(t)=\begin{cases}
	\beta(\beta-1)t^{\beta-2}, \qquad & 0<t<T,\\
	\beta(\beta-1)(-t)^{\beta-2}, \qquad & -T<t<0,\\
	0& \text{ otherwise}.
	\end{cases}
	\end{equation*}
	
	We observe that $ \varphi(u)\in X_0^1 $.
	As a matter of fact, from the definition of $ \varphi $, we have
	\begin{itemize}
	\item $\varphi(u)=0 $ a.e. in $ \mathbb{R}^n\backslash \Omega $;
	\item $ \varphi $ is convex and Lipschitz with Lipschitz constant $ L=\beta T^{\beta-1} $;
	\item  one can calculate that
	\begin{equation}
	\begin{split}
	\|\varphi(u)\|^2_{X_0^1}&=\|\nabla( \varphi(u))\|^2_2+\int\!\!\!\int\nolimits_{\mathbb{R}^{2n}}\frac{|\varphi(u(x))-\varphi(u(y))|^2}{|x-y|^{n+2s}}\, dxdy  \\
	&\leqslant L^2\bigg(\|\nabla u\|^2_2+\int\!\!\!\int\nolimits_{\mathbb{R}^{2n}}\frac{|u(x)-u(y)|^2}{|x-y|^{n+2s}}\, dxdy\bigg)<\infty.
	\label{4.1}
	\end{split}
	\end{equation}
	\end{itemize}
	
	Moreover, for  given  $ x$, $ y\in \mathbb{R}^n $, we deduce from the convexity of~$\varphi$ that
	\begin{equation*}
	\varphi(u(x))-\varphi(u(y))\leqslant \varphi'(u(x))(u(x)-u(y)),
	\end{equation*}
	which implies that
	\begin{equation}
	(-\Delta)^s\varphi(u)\leqslant \varphi'(u)(-\Delta)^su \label{4.2}
	\end{equation}
	in the sense of distribution.
	\vspace{2mm}
	
	\par\noindent \textbf{Step 2}. Now we give an estimate of  $\|\varphi(u)\|_{2_s^*}$.
	\vspace{2mm}	
	For this, we recall that
	\begin{equation*}
	\|u\|^2_{L^{2^*_s}(\mathbb{R}^n)}\leqslant S\int\!\!\!\int\nolimits_{\mathbb{R}^{2n}}\frac{|u(x)-u(y)|^2}{|x-y|^{n+2s}}\, dxdy,
	\end{equation*}
	where $ S$ is a constant depending only on $n$ and $s$.
	
	Hence, by \eqref{4.2}, we have that
	\begin{equation}
	\begin{split}
	\|\varphi(u)\|^2_{2^*_s}&\leqslant S \int\!\!\!\int\nolimits_{\mathbb{R}^{2n}}\frac{|\varphi(u(x))-\varphi(u(y))|^2}{|x-y|^{n+2s}}\, dxdy\\
	&=S \int\!\!\!\int\nolimits_{\mathbb{R}^{2n}}\left[\varphi(u(x))\frac{\varphi(u(x))-\varphi(u(y))}{|x-y|^{n+2s}}-\varphi(u(y))\frac{\varphi(u(x))-\varphi(u(y))}{|x-y|^{n+2s}}\right]\, dxdy\\
	&=S \int\!\!\!\int\nolimits_{\mathbb{R}^{2n}}\left[\varphi(u(x))\frac{\varphi(u(x))-\varphi(u(y))}{|x-y|^{n+2s}}-\varphi(u(x))\frac{\varphi(u(y))-\varphi(u(x))}{|y-x|^{n+2s}}\right]\, dxdy\\
	&=2S\int\!\!\!\int\nolimits_{\mathbb{R}^{2n}}\varphi(u(x))\frac{\varphi(u(x))-\varphi(u(y))}{|x-y|^{n+2s}}\,dxdy\\&
	=2S\int_{\mathbb{R}^{n}}\varphi(u)(-\Delta)^s\varphi(u)\, dx\\
	& \leqslant 2S\int_{\mathbb{R}^{n}}\varphi(u)\, \varphi'(u) \,(-\Delta)^su\, dx.
	\end{split}\label{3.2BIS}
	\end{equation}
	We multiply the first equation of (\ref{eq:g main equation}) by  $ \varphi(u)\varphi'(u) $ and  integrate over $ \mathbb{R}^n $
	\begin{equation}\label{integration of the main equation}
	-\int_{\mathbb{R}^{n}}\varphi(u)\varphi'(u)\Delta u\,dx+\int_{\mathbb{R}^{n}}\varphi(u)\,\varphi'(u)\,(-\Delta)^su\, dx=\int_{\mathbb{R}^{n}}\varphi(u)\varphi'(u)g(x,u)\, dx.
	\end{equation}
	The first term on the left side of \eqref{integration of the main equation} above could be rewritten as
	\begin{equation}\label{rewriting of the first term}
	\begin{split}
	&-\int_{\Omega} \varphi(u)\varphi'(u)\Delta u\, dx= \int_{\Omega}\nabla u\cdot \nabla(\varphi(u)\varphi'(u))\, dx\\
	=&\int_{\Omega}|\nabla u|^2|\varphi'(u)|^2\, dx+\int_{\Omega}|\nabla u|^2\varphi(u)\,\varphi''(u)\,  dx\geqslant0.
	\end{split}
	\end{equation}
        The equations \eqref{rewriting of the first term} and \eqref{integration of the main equation} lead to
	\begin{align}
	\int_{\mathbb{R}^{n}}\varphi(u)\,\varphi'(u)\,(-\Delta)^su  dx
	&\leqslant\int_{\Omega}\varphi(u)\,\varphi'(u)\,g(x,u)\, dx. \label{4.5}
	\end{align}
	Hence, {recalling~\eqref{eq: H 1}} and~\eqref{3.2BIS}, we obtain
	\begin{align*}
	\|\varphi(u)\|^2_{2^*_s}&\leqslant 2S\int_{\Omega}\varphi(u)\varphi'(u)g(x,u)\, dx\\&
	\leqslant 2S \int_{\Omega}\varphi(u)|\varphi'(u)||g(x,u)|\, dx\\
	&\leqslant C \int_{\Omega}\varphi(u)|\varphi'(u)|(1+|u|^{2^*_s-1})\, dx\\
	& =C\bigg(\int_{\Omega}\varphi(u)|\varphi'(u)|\,dx +\int_{\Omega}\varphi(u)|\varphi'(u)||u|^{2^*_s-1}\, dx \bigg),
	\end{align*}
	where the constant $ C$ depends on $ n,s,\Omega$ and $c$.
	
	Using the estimates $ \varphi(u)\leqslant |u|^\beta $, $ |\varphi'(u)|\leqslant \beta |u|^{\beta-1} $ and $ |u\varphi'(u)|\leqslant \beta \varphi(u) $, we have
	\begin{equation}\label{main estimate}
	\bigg(\int_{\Omega}(\varphi(u))^{2^*_s}\,dx\bigg)^{2/2^*_s}\leqslant C\beta\bigg(\int_{\Omega}|u|^{2\beta-1}\,dx +\int_{\Omega}(\varphi(u))^2 |u|^{2^*_s-2}\, dx \bigg).
	\end{equation}
	Notice that $ C $ is a positive constant that does not depend on $ \beta $ and  that the last integral on the right-hand-side of the inequality \eqref{main estimate} is well defined for every $ T >0 $, since
	\begin{align*}
	\int_{\Omega}(\varphi(u))^2 |u|^{2^*_s-2}\, dx &=\int_{\left\{|u|\leqslant T\right\}} (\varphi(u))^2 |u|^{2^*_s-2}\, dx+\int_{\left\{|u|>T\right\}}(\varphi(u))^2 |u|^{2^*_s-2}\, dx\\
	&\leqslant T^{2\beta-2}\int_{\Omega}|u|^{2^*_s}\, dx+(\beta+1) T^{\beta-1}\int_{\Omega} |u|^{2^*_s}\, dx<+\infty.
	\end{align*}
	\vspace{2mm}
	
	\par\noindent \textbf{Step 3}.  We apply  Moser iteration method and the limiting arguments of $L^p$-norm  to prove the result.
	
	\vspace{2mm}
	
	First, we claim if $\beta_1$ is such that $ 2\beta_1-1=2^*_s $,   then  $ u\in L^{2^*_s\beta_1}(\Omega)$. To see this, for any given $ R>0 $, we apply the H\"older inequality  to the last integral in \eqref{main estimate} and get
	\begin{equation}
	\begin{split}
	&\int_{\Omega}(\varphi(u))^2 |u|^{2^*_s-2}\, dx=\int_{\left\{|u|\leqslant R\right\}}(\varphi(u))^2 |u|^{2^*_s-2}\, dx +\int_{\left\{|u|>R\right\}} (\varphi(u))^2 |u|^{2^*_s-2}\, dx \\
	 \leqslant& \int_{\left\{|u|\leqslant R\right\}}\frac{(\varphi(u))^2}{|u|} R^{2^*_s-1}\, dx+\bigg(\int_{\Omega}(\varphi(u))^{2^*_s}\, dx\bigg)^{2/2^*_s}\bigg(\int_{\left\{|u|>R\right\}}|u|^{2^*_s}\, dx\bigg)^{\frac{2^*_s-2}{{2^*_s}}}.
	\end{split}\label{estimate of phi u}
	\end{equation}
	By the Monotone Convergence Theorem, one can choose $R>0$ large enough such that
	\begin{equation}\label{eq:C beta}
	\bigg(\int_{\left\{|u|>R\right\}}|u|^{2^*_s}\,dx\bigg)^{\frac{2^*_s-2}{{2^*_s}}}\leqslant \frac{1}{2C\beta_1},
	\end{equation}
	where $ C $ is the constant in \eqref{main estimate}. Therefore one can reabsorb the last term in \eqref{estimate of phi u} into the left hand side of \eqref{main estimate} to get
	\begin{equation*}
	\bigg(\int_{\Omega}(\varphi(u))^{2^*_s}\,dx\bigg)^{2/2^*_s}\leqslant 2C\beta_1\bigg(\int_{\Omega}|u|^{2^*_s}\, dx+ R^{2^*_s-1}\int_{\Omega}\frac{(\varphi(u))^2}{|u|}\, dx \bigg).
	\end{equation*}
	Now, using $\varphi(u)\leqslant |u|^{\beta_1}$, we get that the terms in the right hand side of the above inequality is bounded and independent of $T$. Sending $ T\rightarrow +\infty $, we obtain that
	\begin{equation}\label{eq:u in L beta1}
	\bigg(\int_{\Omega}|u|^{2^*_s\beta_1}\,dx\bigg)^{2/2^*_s}\leqslant 2C\beta_1\bigg(\int_{\Omega}|u|^{2^*_s}\, dx+ R^{2^*_s-1}\int_{\Omega}|u|^{2^*_s}\, dx \bigg)<+\infty,
	\end{equation}
	which proves the claim.
	\vspace{2mm}
	
	Next, we will find an increasing unbounded sequence $ \beta_m $  such that
	\[ u\in L^{2_s^*\beta_{m}}(\Omega), \quad \forall m>1. \]
	To this end, let us suppose that~$ \beta>\beta_1 $. Thus, using that $ \varphi(u)\leqslant |u|^\beta $ in the right hand side of $ \eqref{main estimate} $ and letting $ T\rightarrow \infty $ we get
	\begin{equation}
	\bigg(\int_{\Omega}|u|^{2^*_s\beta}\,dx\bigg)^{2/2^*_s}\leqslant C\beta \bigg(\int_{\Omega} |u|^{2\beta-1}\,dx +\int_{\Omega} |u|^{2\beta+2^*_s-2}\, dx\bigg). \label{rough estimate}
	\end{equation}
We also remark that
	\begin{equation}\label{estimate of 2beta-1-norm}
	\begin{split}
	\int_{\Omega}|u|^{2\beta-1}\,dx&\leqslant \bigg(\int_{\Omega}|u|^{2\beta+2^*_s-2}\,dx\bigg)^{\frac{2\beta-1}{2\beta+2^*_s-2}}|\Omega|^{\frac{2^*_s-1}{2\beta+2^*_s-2}}\\
	&\leqslant \frac{2\beta-1}{2\beta+2^*_s-2}\int_{\Omega}|u|^{2\beta+2^*_s-2}\,dx+\frac{2^*_s-1}{2\beta+2^*_s-2}|\Omega|\\
	&\leqslant \int_{\Omega}|u|^{2\beta+2^*_s-2}\,dx+|\Omega|.
	\end{split}
	\end{equation}
Hence, by combining \eqref{rough estimate} with \eqref{estimate of 2beta-1-norm}, we conclude that
	\begin{align*}
	\bigg(\int_{\Omega}|u|^{2^*_s\beta}\,dx\bigg)^{2/2^*_s}&\leqslant  C\beta \left(|\Omega| +2\int_{\Omega} |u|^{2\beta+2^*_s-2}\, dx\right)\\
	&\leqslant 2C\beta(|\Omega|+1)\left(1+\int_{\Omega} |u|^{2\beta+2^*_s-2}\, dx\right).
	\end{align*} Moreover,
	by the formula $ (a+b)^2\leqslant 2(a^2+b^2) $, we see that
	\begin{align*}
	\left(1+\int_{\Omega} |u|^{2^*_s\beta}\, dx\right)^2&\leqslant 2+2\left[2C\beta(|\Omega|+1)\left(1+\int_{\Omega} |u|^{2\beta+2^*_s-2}\, dx \right)\right]^{2^*_s}.
	\end{align*}
	Therefore,
	\begin{equation}\label{estimate of iteration}
	\bigg(1+\int_{\Omega}|u|^{2^*_s\beta}\,dx\bigg)^{\frac{1}{2^*_s(\beta-1)}}\leqslant (C\beta)^{\frac{1}{2(\beta-1)}}\bigg(1+\int_{\Omega} |u|^{2\beta+2^*_s-2}\, dx\bigg)^{\frac{1}{2(\beta-1)}},
	\end{equation}
	where $ C$ is renamed independently of $ \beta.$
	
	For $ m\geqslant 1, $ we define $ \beta_{m+1} $ such that
	\begin{equation*}
	2\beta_{m+1}+2^*_s-2=2^*_s\beta_{m}.
	\end{equation*}
	Thus,
	\begin{equation*}
	\beta_{m+1}-1=\left(\frac{2^*_s}{2}\right)^m(\beta_1-1)
	\end{equation*}
	and \eqref{estimate of iteration} becomes
	\begin{equation*}
	\bigg(1+\int_{\Omega}|u|^{2^*_s\beta_{m+1}}\,dx\bigg)^{\frac{1}{2^*_s(\beta_{m+1}-1)}}\leqslant (C\beta_{m+1})^{\frac{1}{2(\beta_{m+1}-1)}}\bigg(1+\int_{\Omega} |u|^{2^*_s\beta_{m}}\, dx\bigg)^{\frac{1}{2^*_s(\beta_{m}-1)}}.
	\end{equation*}
	We now define $ C_{m+1}:=C\beta_{m+1} $ and
	\begin{equation*}
	A_m:=\bigg(1+\int_{\Omega} |u|^{2^*_s\beta_{m}}\, dx\bigg)^{\frac{1}{2^*_s(\beta_{m}-1)}}.
	\end{equation*}
In particular, note that $A_1= \bigg(1+\int_{\Omega} |u|^{2^*_s\beta_{1}}\, dx\bigg)^{\frac{1}{2^*_s(\beta_{1}-1)}}$ is bounded by \eqref{eq:u in L beta1}.
	
	Now we claim that there exists a constant $ C_0>0 $ independent of $m$, such that
	\begin{equation}\label{DRI}
	A_{m+1}\leqslant \prod_{k=2}^{m+1}C_k^{\frac{1}{2(\beta_{k}-1)}}A_1\leqslant C_0 A_1.
	\end{equation}
	We stress that once~\eqref{DRI} is established,
then by the H\"older inequality, we conclude that $ u\in L^p(\Omega)$, for every~$ p\in[1,+\infty)$. Furthermore,
	a limiting argument implies that
	\begin{equation*}
	\|u\|_{\infty}\leqslant C_0 {A_1}<+\infty,
	\end{equation*}
	which would complete the proof of Theorem~\ref{th: regularity}.
	
Hence, it remains to check~\eqref{DRI}. To this end, we write $\bar q=\frac{2}{2^*_s}<1 $, and  observe that
	\begin{equation*}
	\beta_{m+1}=\left(\frac{1}{\bar q}\right)^m(\beta_1-1)+1
	=\left(\frac{1}{\bar q}\right)^{m+1}-\frac{1}{2}\left(\frac{1}{\bar q}\right)^m+1\leqslant2\left(\frac{1}{\bar q}\right)^{m+1},
	\end{equation*}
	thus, $ C_k=C\beta_{k}\leqslant 2C\left(\frac{1}{\bar q}\right)^k $, and
	\begin{align*}
	\prod_{k=2}^{m+1}C_k^{\frac{1}{2(\beta_{k}-1)}}&\leqslant \prod_{k=2}^{m+1} \bigg(2C\left(\frac{1}{\bar q}\right)^k\bigg)^{\frac{\bar{q}^{k-1}}{2(\beta_1-1)}}\\
	&=\bigg[\bigg(2C\left(\frac{1}{\bar q}\right)^{(m+1)}\bigg)^{\bar{q}^m}\bigg(2C\left(\frac{1}{\bar q}\right)^{m}\bigg)^{\bar{q}^{m-1}}\cdots \bigg(2C\left(\frac{1}{\bar q}\right)^{2}\bigg)^{\bar q}\bigg]^{\frac{1}{2(\beta_1-1)}}.
	\end{align*}
	We consider
	\begin{align*}
	\bigg(2C\left(\frac{1}{\bar q}\right)^{(m+1)}\bigg)^{\bar{q}^m}\bigg(2C\left(\frac{1}{\bar q}\right)^{m}\bigg)^{\bar{q}^{m-1}}\cdots \bigg(2C\left(\frac{1}{\bar q}\right)^{2}\bigg)^{\bar q}=(2C)^{\sum\limits_{k=1}^{m} \bar{q}^k}\bigg((\frac{1}{\bar q})^{\bar{q}^m(m+1)+\bar{q}^{m-1}m+\cdots +2 \bar{q}}\bigg),
	\end{align*}
	where
	\begin{equation*}
	\sum\limits_{k=1}^{m}\bar{q}^k=\frac{\bar{q}(1-\bar{q}^m)}{1-\bar q}\leqslant\frac{\bar q}{1-\bar q},
	\end{equation*}
	and
	\begin{equation*}
	0<\bar{q}^m(m+1)+\bar{q}^{m-1}m+\cdots +2\bar{q}=\frac{\bar{q}(1 - \bar{q}^m)}{(1- \bar{q})^2}+\frac{\bar q}{1- \bar{q}}-\frac{\bar{q}^{m+1}(m+1)}{1- \bar{q}}<\frac{\bar q}{(1- \bar{q})^2}+\frac{\bar{q}}{1-\bar{q}},
	\end{equation*}
	which implies that there exists $ C_0>0 $ independent of $ m $ such that
	\begin{equation*}
	\prod_{k=2}^{m+1}C_k^{\frac{1}{2(\beta_{k}-1)}}\leqslant C_0,
	\end{equation*}	
where $ C_0 $ depends on $ c,n,s,\Omega$.	The proof is completed.
\end{proof}

\bigskip

\section{$ C^{1,\alpha} $-regularity }\setcounter{equation}{0}

In order to obtain the $ C^{1,\alpha} $-regularity of weak solutions of \eqref{eq:g main equation}, it suffices to prove a general $ W^{2,p} $-regularity result  for the mixed operator $ \mathcal{L} $ in Theorem~\ref{theorem w 2p}.

For any $ u\in W^{2,p}(\Omega) $ and $ f\in L^p(\Omega) $, we always think of  them as the functions defined on the whole space $\mathbb{R}^n$ such that they vanish outside $ \Omega$.

\subsection{$ W^{2,p} $-regularity}

We begin by recalling the following classical results about the $W^{2,p}$-regularity of the operator $ -\Delta$.

Let $ p $ be given. Then, there exists $ \lambda_0\geqslant 0 $ such that the problem
\begin{equation}\label{eq: classical equation lambda}
-\Delta u_\lambda+\lambda u_\lambda=f
\end{equation}
has a unique solution $ u_\lambda\in W^{2,p}(\Omega)\cap W^{1,p}_0(\Omega) $ for any $ \lambda\geqslant\lambda_0, f\in L^p(\Omega) $ and
\begin{align}
	\|u_\lambda\|_{W^{2,p}(\Omega)}\leqslant C\|f\|_{p} \label{eq:lapalce estimate 1}\\
{	(\lambda-\lambda_0)\|u_\lambda\|_{p}\leqslant C\|f\|_{p}}\label{eq:lapalce estimate 2}
\end{align}
where the positive constant  $ C$ is independent of  $ u_\lambda$ and $ \lambda$.

According to the above results, the proof of  Theorem~\ref{theorem w 2p} will be divided into the following subsections:
Subsection~\ref{subsection 4.11} is devoted to the basic $ L^p $ estimates of the operator $ (-\Delta)^s$;  in Subsection~\ref{section 4.12},
we apply the fixed point theorem  to obtain that the problem
	\begin{equation*}
	-\Delta u+\lambda u=f-(-\Delta)^su, \qquad \text{in } \Omega
	\end{equation*}
	has a unique solution $ u\in W^{2,p}(\Omega)\cap W^{1,p}_0(\Omega) $ for $ \lambda>0$ large enough; in Subsection~\ref{subsection 4.13}, thanks to the Maximum Principle  in \cite[Theorem 1.2]{MR4387204} and the
	bootstrap method, one can obtain the result as desired.

\subsubsection{$ L^p $ estimate of $ (-\Delta)^s $}\label{subsection 4.11}

We introduce the following Extension Theorem which is useful for the $ L^p $ estimate of $ (-\Delta)^s$.
\begin{Theorem}\label{theorem extension}[See e.g. \cite[Chapter 7]{GilbargTrudingerbook}]
	Let $ \Omega $ be a $ C^{k-1,1} $ domain in $ \mathbb{R}^n $, $ k\geqslant 1$. Then
	\begin{itemize}
		\item[(i)] $ C^\infty(\overline{\Omega}) $ is dense in $ W^{k,p}(\Omega), 1\leqslant p<\infty$.
		\item[(ii)] For any open set $ \Omega'\supset\supset \Omega$, there exists a bounded linear extension operator $ E $ from $  W^{k,p}(\Omega) $ into $  W^{k,p}_0(\Omega') $ such that $ Eu=u $ in $ \Omega $ and
		\begin{equation*}
		\|Eu\|_{ W^{k,p}(\Omega')}\leqslant C\|u\|_{ W^{k,p}(\Omega)} \qquad \text{for all $ u\in W^{k,p}(\Omega) $},
		\end{equation*}
		where $ C=C(k,\Omega,\Omega')$.
	\end{itemize}
\end{Theorem}

Now we state some useful interpolation results.
 \begin{Lemma}\label{lemma Lp estimate about fractional operator}
 Let $ \Omega $ be a $ C^{1,1} $ domain in $ \mathbb{R}^n $ and $ s\in(0,\frac{1}{2}) $. Then the operator $ (-\Delta)^s\in \mathscr{L}(W^{2,p}(\Omega)\cap W_0^{1,p}(\Omega),L^p(\Omega)) $ with  $ p>1 $, and we have
 \begin{equation}\label{eq: Lp estimate of fractional s}
 	\|(-\Delta)^s u\|_{p}\leqslant C\left[\epsilon\|u\|_{W^{2,p}(\Omega)}+\tau(\epsilon)\|u\|_{p}\right] ,  \qquad\text{ for every $\epsilon>0$},
 \end{equation}where $ C $ is a constant independent of $ u$.
 \end{Lemma}
\begin{proof}
Let $u\in W_0^{1,p}(\Omega) $. We split the term $(-\Delta)^su(x)$ into three parts and estimate their $L^p$-norms one by one.
\begin{align*}&
	(-\Delta)^su(x)\\ =\,&\int_{\mathbb{R}^n}\frac{u(x)-u(y)}{|x-y|^{n+2s}}\, dy\\
	=\,&\int_{\{y\in\Omega\}\cap\{|x-y|\leqslant 1\}}\frac{u(x)-u(y)}{|x-y|^{n+2s}}\, dy+\int_{\{y\notin\Omega\}\cap\{|x-y|\leqslant 1\}}\frac{u(x)}{|x-y|^{n+2s}}\, dy+\int_{\{|x-y|> 1\}}\frac{u(x)-u(y)}{|x-y|^{n+2s}}\, dz\\
	=\,&\uppercase\expandafter{\romannumeral 1}+\uppercase\expandafter{\romannumeral 2}+\uppercase\expandafter{\romannumeral 3}.
\end{align*}
Now, using the change of variable $ z=y-x  $, and applying the Extension Theorem~\ref{theorem extension}  and the H\"{o}lder inequality, we see that
\begin{align*}
	\|\uppercase\expandafter{\romannumeral 1}\|_{p}&\leqslant\bigg(\int_{\Omega}\, dx\bigg[\int_{\{x+z\in\Omega\}\cap\{|z|\leqslant 1\}}\frac{|u(x)-u(x+z)|}{|z|^{n+2s}}\, dz\bigg]^p\bigg)^{1/p}\\
	&=\bigg(\int_{\Omega}\, dx \bigg[\int_{\{x+z\in\Omega\}\cap\{|z|\leqslant 1\}} \frac{\int_{0}^{1}|\nabla u(x+\theta z)|\,d\theta}{|z|^{n+2s-1}}\, dz\bigg]^p\bigg)^{1/p}\\
	&\leqslant \bigg(\int_{\mathbb{R}^n}\, dx \bigg[\int_{\{|z|\leqslant 1\}}\int_{0}^{1}  \frac{|\nabla Eu(x+\theta z)|^p}{|z|^{n+2s-1}}\,d\theta\, dz\bigg]\bigg(\int_{\{|z|\leqslant 1\}}\int_{0}^{1}\frac{1}{|z|^{n+2s-1}}\,d\theta\, dz\bigg)^{\frac{p}{q}}\bigg)^{1/p}\\
	&\leqslant \|Eu\|_{W^{1,p}(\mathbb{R}^n)}\bigg(\int_{\{|z|\leqslant 1\}}\frac{1}{|z|^{n+2s-1}}\, dz\bigg)
	\leqslant C(n,s,p) \|u\|_{W^{1,p}(\Omega)}.
\end{align*}
Also, we estimate
\begin{align*}
	\|\uppercase\expandafter{\romannumeral 3}\|_{p}&\leqslant\bigg(\int_{\mathbb{R}^n}\, dx\bigg[\int_{\{|z|> 1\}}\frac{|u(x)|+|u(x+z)|}{|z|^{n+2s}}\, dz\bigg]^p\bigg)^{1/p}\\
	&\leqslant C(p)\bigg(\int_{\mathbb{R}^n}|Eu(x)|^p+|Eu(x+z)|^p\, dx\bigg)^{1/p}\bigg(\int_{\{|z|> 1\}}\frac{1}{|z|^{n+2s}}\, dz\bigg)\\
	&\leqslant C(p,s)\|Eu\|_{L^p(\mathbb{R}^n)}
	\leqslant C(p,s,\Omega)\|u\|_{W^{1,p}(\Omega)}.
\end{align*}
Set $ d(x):= $  dist$(x,\partial\Omega)$ for any $ x\in\Omega $. Then note that for  $ u\in  W_0^{1,p}(\Omega) $, we have
\begin{align*}
\|\uppercase\expandafter{\romannumeral 2}\|_{p}&\leqslant\bigg(\int_{\Omega}\, dx\bigg[\int_{\{x+z\notin\Omega\}\cap\{|z|\leqslant 1\}}\frac{|u(x)|}{|z|^{n+2s-1}d(x)}\, dz\bigg]^p\bigg)^{1/p}
\leqslant C(s,n,p)\bigg(\int_{\Omega} \left|\frac{u(x)}{d(x)}\right|^p\, dx\bigg)^{1/p}\\
&\leqslant C(n,s,p)\bigg(\int_{\Omega}\int_{0}^{1}|\nabla u(x+\theta(x-x_0))|^p\,dt\,dx\bigg)^{1/p}
\leqslant C(n,s,p,\Omega)\|u\|_{W^{1,p}(\Omega)},
\end{align*}
where $ x_0\in\partial\Omega $ such that $ d(x)= $dist$(x,x_0)$. Here, the constant $ C$ changes from line to line, but remains independent of $ u$.

As a consequence, the Sobolev interpolation inequality implies that for  every $\epsilon>0$, we have
\begin{equation*}
	\|(-\Delta)^s u\|_{p}\leqslant C\bigg[\epsilon\|u\|_{W^{2,p}(\Omega)}+\tau(\epsilon)\|u\|_{p}\bigg]. \qedhere
\end{equation*}
\end{proof}

Next, we extend the preceding result to the case $ s=\frac{1}{2} $.

\begin{Lemma}\label{lemma Lp estimate about 1/2 operator}
	 Let $ \Omega $ be a $ C^{1,1} $ domain in $ \mathbb{R}^n $ and $ s=\frac{1}{2} $. Then the operator $ (-\Delta)^{\frac{1}{2}}\in \mathscr{L}(W^{2,p}(\Omega)\cap W_0^{1,p}(\Omega),L^p(\Omega)) $ with  $ p>1 $, and
	\begin{equation}\label{eq: Lp estimate of 1/2}
{	\|(-\Delta)^{\frac{1}{2}} u\|_{p}\leqslant C\bigg[{o(\epsilon)}\|u\|_{W^{2,p}(\Omega)}+\tau(\epsilon)\|u\|_{W^{1,p}(\Omega)}\bigg], \qquad \text{for every $\epsilon>0$}}
	\end{equation}
	where $ C$ is a constant independent of $u$,
	{ $ o(\epsilon)\rightarrow 0 $ as $ \epsilon\rightarrow 0$}, and $ \tau(\epsilon)$ is unbounded as $\epsilon\rightarrow 0$.
\end{Lemma}
\begin{proof}
	\par\noindent \textbf{Case 1.} $ p>n $. Then $ u\in C^{1}(\overline{\Omega}) $.
	
	Let $ u\in W^{2,p}(\Omega)\cap W_0^{1,p}(\Omega)  $. We split the term $-(-\Delta)^{\frac{1}{2}}u $  into five parts and estimate their $L^p$-norms one by one.
	\begin{align*}
	-(-\Delta)^{\frac{1}{2}}u(x)&=\frac{1}{2}\int_{\mathbb{R}^n}\frac{u(x+z)+u(x-z)-2u(x)}{|z|^{n+1}}\, dz\\
	&=\frac{1}{2}\int_{\Omega_{\uppercase\expandafter{\romannumeral 1}}} \frac{u(x+z)+u(x-z)-2u(x)}{|z|^{n+1}}\, dz\\
	&\quad +\frac{1}{2}\int_{\Omega_{\uppercase\expandafter{\romannumeral 2}}} \left[ \frac{u(x+z)-u(x)}{|z|^{n+1}}+\frac{-u(x)}{|z|^{n+1}}\right]\, dz\\
	&\quad +\frac{1}{2}\int_{\Omega_{\uppercase\expandafter{\romannumeral 3}}}\left[\frac{u(x-z)-u(x)}{|z|^{n+1}}+\frac{-u(x)}{|z|^{n+1}}\right]\, dz\\
	&\quad +\int_{\Omega_{\uppercase\expandafter{\romannumeral 4}}}\frac{-u(x)}{|z|^{n+1}}\, dz+ \int_{\{|z|> 1\}}\frac{u(x+z)+u(x-z)-2u(x)}{|z|^{n+1}}\, dz\\
	&=\uppercase\expandafter{\romannumeral 1}+\uppercase\expandafter{\romannumeral 2}+\uppercase\expandafter{\romannumeral 3}+\uppercase\expandafter{\romannumeral 4}+\uppercase\expandafter{\romannumeral 5},
	\end{align*}
	where \begin{equation*}
		\begin{split}
		\Omega_{\uppercase\expandafter{\romannumeral 1}}&:={\{x+ z\in\Omega\}\cap\{x- z\in\Omega\}\cap\{|z|\leqslant 1\}},\\
		\Omega_{\uppercase\expandafter{\romannumeral 2}}&:={\{x+ z\in\Omega\}\cap\{x- z\notin\Omega\}\cap\{|z|\leqslant 1\}},\\
		\Omega_{\uppercase\expandafter{\romannumeral 3}}&:=
		{\{x+ z\notin\Omega\}\cap\{x- z\in\Omega\}
	\cap\{|z|\leqslant 1\}},\\
\Omega_{\uppercase\expandafter{\romannumeral 4}}&:={\{x+ z\notin\Omega\}\cap\{x- z\notin\Omega\}
\cap\{|z|\leqslant 1\}}.
		\end{split}
	\end{equation*}

	\textbf{(i). Estimate of the term $ \uppercase\expandafter{\romannumeral 1}$ above.}		
	Given~$\epsilon\in(0,1) $, set $ \uppercase\expandafter{\romannumeral 1}=\uppercase\expandafter{\romannumeral 1}_{|z|\leqslant\epsilon}+\uppercase\expandafter{\romannumeral 1}_{\epsilon<|z|\leqslant1} $, and then we have
	\begin{align*}
	\|\uppercase\expandafter{\romannumeral 1}_{|z|\leqslant\epsilon}\|_{p}&=\frac{1}{2}\bigg(\int_{\Omega}\, dx\bigg[\int_{\Omega_{\uppercase\expandafter{\romannumeral 1}}\cap\{|z|\leqslant\epsilon\}}\frac{\int_{0}^{1}\,d\theta\int_{0}^{\theta}\, d\theta'D_{ij}u(x+\theta'z)z_iz_j}{|z|^{n+1}}\, dz\bigg]^p\bigg)^{1/p}\\
	&\leqslant \bigg(\int_{\mathbb{R}^n}\, dx \bigg[\int_{\{|z|\leqslant \epsilon\}}\int_{0}^{1} \int_{0}^{\theta} \frac{|D^2 (Eu(x+\theta' z))|^p}{|z|^{n-1}}\, d\theta'\,d\theta\, dz\bigg]\bigg(\int_{\{|z|\leqslant \epsilon\}}\frac{1}{|z|^{n-1}}\, dz\bigg)^{\frac{p}{q}}\bigg)^{1/p}\\
	&\leqslant \|Eu\|_{W^{2,p}(\mathbb{R}^n)}\bigg(\int_{\{|z|\leqslant \epsilon\}}\frac{1}{|z|^{n-1}}\, dz\bigg)
	\leqslant C(n,p) o(\epsilon) \|u\|_{W^{2,p}(\Omega)},
	\end{align*}
	\begin{align*}
	\|\uppercase\expandafter{\romannumeral 1}_{\epsilon<|z|\leqslant1}\|_{p}&=\bigg(\int_{\Omega}\, dx\bigg[\int_{\Omega_{\uppercase\expandafter{\romannumeral 1}}\cap\{\epsilon<|z|\leqslant 1\}}\frac{u(x+z)-u(x)}{|z|^{n+1}}\, dz\bigg]^p\bigg)^{1/p}\\
	&\leqslant \bigg(\int_{\mathbb{R}^n}\, dx \bigg[\int_{\epsilon<|z|\leqslant1}\int_{0}^{1}  \frac{|\nabla Eu(x+\theta z)|^p}{|z|^{n}}\,d\theta\, dz\bigg]\bigg(\int_{\epsilon<|z|\leqslant1}\frac{1}{|z|^{n}}\, dz\bigg)^{\frac{p}{q}}\bigg)^{1/p}\\
	&\leqslant \|Eu\|_{W^{1,p}(\mathbb{R}^n)}\bigg(\int_{\epsilon<|z|\leqslant1}\frac{1}{|z|^{n}}\, dz\bigg)
	\leqslant C(n,p) \tau(\epsilon) \|u\|_{W^{1,p}(\Omega)},
	\end{align*}
	where $ o(\epsilon)\rightarrow 0$ as $ \epsilon\rightarrow 0$, and $ \tau(\epsilon)$ is arbitrary and unbounded, as $ \epsilon\rightarrow 0$. Gathering the above estimates, we obtain that
	\begin{equation*}
	\|\uppercase\expandafter{\romannumeral 1}\|_{p}\leqslant C(n,p)\bigg[o(\epsilon) \|u\|_{W^{2,p}(\Omega)}+\tau(\epsilon) \|u\|_{W^{1,p}(\Omega)}\bigg].
	\end{equation*}
	\smallskip
	
	\textbf{(ii). Estimates of the terms  $ \uppercase\expandafter{\romannumeral 2 }\sim \uppercase\expandafter{\romannumeral 4}$.}	
	Set $ d(x):= $  dist$(x,\partial\Omega)$ for any $ x\in\Omega $. Noting that $ u\in C^{1}(\overline\Omega),  u|_{\partial\Omega}=0 $, and for all~$ p>n $, there exists $ \alpha>0 $ such that $ p<\frac{1}{\alpha} $, then, given $ \epsilon\in(0,1) $, we see  that
	
	\begin{equation*}
	\int_{\Omega_{\uppercase\expandafter{\romannumeral 2}}\cap\{|z|\leqslant\epsilon\}}\frac{u(x+z)-u(x)}{|z|^{n+1}}\, dz\leqslant	\int_{\Omega_{\uppercase\expandafter{\romannumeral 2}}\cap\{|z|\leqslant\epsilon\}}\frac{\|\nabla u(x)\|_{\infty}}{|z|^{n-\alpha} d(x)^\alpha} \, dz,
	\end{equation*}
	and
	\begin{equation*}
	\int_{\Omega_{\uppercase\expandafter{\romannumeral 2}}\cap\{|z|\leqslant\epsilon\}}\frac{-u(x)}{|z|^{n+1}}\, dz\leqslant \int_{\Omega_{\uppercase\expandafter{\romannumeral 2}}\cap\{|z|\leqslant\epsilon\}}\frac{\|\nabla u(x)\|_{\infty}}{|z|^{n-\alpha} d(x)^\alpha} \, dz.
	\end{equation*}
	Furthermore, using the well-known Sobolev inequality, we obtain that 	
	\begin{align*}
	\|\uppercase\expandafter{\romannumeral 2}_{|z|\leqslant\epsilon}\|_{p}&\leqslant \bigg(\int_{\Omega} \,dx \bigg[\int_{\Omega_{\uppercase\expandafter{\romannumeral 2}}\cap\{|z|\leqslant\epsilon\}}\frac{\|\nabla u(x)\|_{\infty}}{|z|^{n-\alpha} d(x)^\alpha} \, dz\bigg]^p\bigg)^{1/p}
	\leqslant C(n,p)o(\epsilon)\|u\|_{W^{2,p}(\Omega)},
	\end{align*}
	and by the estimate of $ \|\uppercase\expandafter{\romannumeral 1}_{\epsilon<|z|\leqslant1}\|_{p} $, we have	
	\begin{equation*}
	\|\uppercase\expandafter{\romannumeral 2}_{\epsilon<|z|\leqslant 1}\|_{p}\leqslant C(n,p)\tau(\epsilon)\|u\|_{W^{1,p}(\Omega)}.
	\end{equation*}
	Thus,
	\begin{equation*}
	\|\uppercase\expandafter{\romannumeral 2}\|_{p}\leqslant C(n,p)\bigg[o(\epsilon) \|u\|_{W^{2,p}(\Omega)}+\tau(\epsilon) \|u\|_{W^{1,p}(\Omega)}\bigg].
	\end{equation*}
	
	By a similar argument  to the estimate of $ \|\uppercase\expandafter{\romannumeral 2}\|_{p} $, we see that
	\begin{equation*}
	\|\uppercase\expandafter{\romannumeral 3}\|_{p}\leqslant C(n,p)\bigg[o(\epsilon) \|u\|_{W^{2,p}(\Omega)}+\tau(\epsilon) \|u\|_{W^{1,p}(\Omega)}\bigg],
	\end{equation*}
	and
	\begin{equation*}
	\|\uppercase\expandafter{\romannumeral 4}\|_{p}\leqslant C(n,p)\bigg[o(\epsilon) \|u\|_{W^{2,p}(\Omega)}+\tau(\epsilon) \|u\|_{W^{1,p}(\Omega)}\bigg].
	\end{equation*}
	\smallskip
	
	\textbf{(iii). Estimate of the term  $  \uppercase\expandafter{\romannumeral 5} $.}	
	Thanks to $ u=0, $ a.e. in $ \mathbb{R}^n\backslash\Omega $, we get
	\begin{align*}
	\|\uppercase\expandafter{\romannumeral 5}\|_{p}&\leqslant\frac{1}{2}\bigg(\int_{\mathbb{R}^n}\, dx\bigg[\int_{\{|z|> 1\}}\frac{2|u(x)|+|u(x+z)|+|u(x-z)|}{|z|^{n+1}}\, dz\bigg]^p\bigg)^{1/p}\\
	&\leqslant C(p,n)\|u\|_{L^p(\mathbb{R}^n)}\\
	&\leqslant C(p,n)\|u\|_{W^{1,p}(\Omega)}.
	\end{align*}
	As a consequence of all these estimates, we obtain \eqref{eq: Lp estimate of 1/2} as desired.
	\medskip
	
	\par\noindent \textbf{Case 2.} $ p=n $. Then $ u\in C^{0,1}(\overline\Omega) $.  The Sobolev inequality  implies that
	\begin{equation*}
	\|\nabla u(x)\|_{\infty}\leqslant C(p,n)\|u\|_{W^{2,n}(\Omega)}.
	\end{equation*}
	Thus, \eqref{eq: Lp estimate of 1/2} follows from the \textbf{Case 1}.
	
	\medskip
	
	\par\noindent \textbf{Case 3.} $ p<n $. According to the proof of \textbf{Case 1}, we only need to verify the estimate of $ \uppercase\expandafter{\romannumeral 2}_{|z|\leqslant \epsilon} $.
	Next, using the H\"{o}lder  and Sobolev inequalities, we find
	\begin{align*}
	&\bigg(\int_{\Omega}\, dx \bigg[\int_{\Omega_{\uppercase\expandafter{\romannumeral 2}}\cap\{|z|\leqslant \epsilon\}}\frac{u(x+z)-u(x)}{|z|^{n+1}}\, dz\bigg]^p\bigg)^{1/p}\\
	\leqslant&\,\bigg(\int_{\Omega}\, dx\bigg[	\int_{\Omega_{\uppercase\expandafter{\romannumeral 2}}\cap\{|z|\leqslant \epsilon\}}\frac{\int_{0}^{1}| \nabla u(x+\theta z)|\, d\theta}{|z|^{n-\alpha} d(x)^\alpha} \, dz\bigg]^p\bigg)^{1/p}\\
	\leqslant&\, \bigg(\int_{\mathbb{R}^n}\, dx\bigg[	\int_{\Omega_{\uppercase\expandafter{\romannumeral 2}}\cap\{|z|\leqslant \epsilon\}}\frac{\int_{0}^{1}| \nabla Eu(x+\theta z)|\, d\theta}{|z|^{n-\alpha} [d(x)^\alpha\chi_{\Omega}+(d(x)+1)\chi_{\mathbb{R}^n\backslash\Omega}]} \, dz\bigg]^p\bigg)^{1/p}\\
	\leqslant&\,  \left\|\frac{|\nabla Eu|}{ [d(x)^\alpha\chi_{\Omega}+(d(x)+1)\chi_{\mathbb{R}^n\backslash\Omega}]}\right\|_{L^p(\mathbb{R}^n)}\bigg(\int_{\{|z|\leqslant \epsilon\}}\frac{1}{|z|^{n-\alpha}}\, dz\bigg)\\
	\leqslant&\, \epsilon^\alpha\bigg(\int_{\mathbb{R}^n}|\nabla Eu|^{\frac{pq}{q-p}}\,dx\bigg)^{\frac{q-p}{pq}}\bigg(\int_{\mathbb{R}^n}\frac{1}{[d(x)^\alpha\chi_{\Omega}+(d(x)+1)\chi_{\mathbb{R}^n\backslash\Omega}]^{q}}\, dx\bigg)^{1/q}, \text{ for } p<n<q<\frac{1}{\alpha}\\
	\leqslant&\, C(n,p,\Omega) \epsilon^\alpha \bigg(\int_{\mathbb{R}^n}|\nabla Eu|^{\frac{pq}{q-p}}\,dx\bigg)^{\frac{q-p}{pq}}
	\leqslant C(n,p,\Omega) \epsilon^\alpha \|Eu\|_{W^{2,p}(\mathbb{R}^n)}
	\leqslant C(n,p,\Omega)o(\epsilon) \|u\|_{W^{2,p}(\Omega)}.
	\end{align*}
	Also,
	\begin{align*}
	&\bigg(\int_{\Omega} \, dx\bigg[\int_{\Omega_{\uppercase\expandafter{\romannumeral 2}}\cap\{|z|\leqslant \epsilon\}}\frac{-u(x)}{|z|^{n+1}}\, dz\bigg]^p\bigg)^{1/p}\\
	\leqslant\,& \bigg(\int_{\Omega}\, dx\bigg[\int_{\Omega_{\uppercase\expandafter{\romannumeral 2}}\cap\{|z|\leqslant \epsilon\}}\frac{|u(x)|}{|z|^{n-\alpha} d(x)^{1+\alpha}} \, dz\bigg]^p\bigg)^{1/p}
	\leqslant \epsilon^\alpha \bigg(\int_{\Omega}\left|\frac{u(x)}{ d(x)}\right|^p \frac{1}{d(x)^{\alpha p}}\, dx\bigg)^{1/p}\\
	\leqslant\,& C(n,p,\Omega) \epsilon^\alpha \bigg(\int_{\Omega} \left|\frac{u(x)}{ d(x)}\right|^{\frac{pq}{q-p}}\, dx\bigg)^{\frac{q-p}{pq}}, \text{ for } p<n<q<\frac{1}{\alpha}\\
	\leqslant\,&  C(n,p,\Omega) \epsilon^\alpha \bigg(\int_{\mathbb{R}^n}|\nabla Eu|^{\frac{pq}{q-p}}\,dx\bigg)^{\frac{q-p}{pq}}
	\leqslant C(n,p,\Omega)o(\epsilon) \|u\|_{W^{2,p}(\Omega)}.
	\end{align*}
	Combining the above two estimates, we get
	\begin{align*}
	\|\uppercase\expandafter{\romannumeral 2}_{|z|\leqslant\epsilon}\|_{p}
	&\leqslant C(n,p,\Omega)o(\epsilon)\|u\|_{W^{2,p}(\Omega)}.
	\end{align*}
	Therefore,  we obtain \eqref{eq: Lp estimate of 1/2} for all cases.
\end{proof}

We observe that Lemmata~\ref{lemma Lp estimate about fractional operator} and~\ref{lemma Lp estimate about 1/2 operator}
cover the range~$s\in\left(0,\frac12\right]$ which is of interest in
the statements of Theorems~\ref{theorem C^1 alpha} and~\ref{theorem w 2p}.
For completeness, though it will not be really utilized in this paper (apart from the comment in Remark~\ref{REMARK for some beta}),
we also point out a similar result when~$s\in\left(\frac{1}{2},1\right)$.

\begin{Lemma}\label{Lp estimate for some p}
 	 Let $ \Omega $ be a $ C^{1,1} $ domain in $ \mathbb{R}^n $ and $ s\in\left(\frac{1}{2},1\right) $. Then for $ n<p<\frac{n}{2s-1}$,  we have
 	 \begin{equation*}
 	 \|(-\Delta)^s u\|_{p}\leqslant C\left[o(\epsilon)\|u\|_{W^{2,p}(\Omega)}+\tau(\epsilon)\|u\|_{W^{1,p}(\Omega)}\right] ,  \qquad\text{ for every $\epsilon>0$},
 	 \end{equation*}where $C$ is a constant independent of $u$.
 \end{Lemma}
\begin{proof}
	It suffices to verify the estimate of $ \uppercase\expandafter{\romannumeral 2}_{|z|\leqslant \epsilon} $. Let  $\Omega_\epsilon:=\{x\in\Omega:d(x)>\epsilon\} $ and $ \gamma$ be a positive  constant satisfying $ n<p<\frac{n}{\gamma}<\frac{n}{2s-1} $. Note that
	\begin{equation*}
		\bigg(\int_{\Omega_\epsilon}\, dx \bigg[\int_{\Omega_{\uppercase\expandafter{\romannumeral 2}}\cap\{|z|\leqslant \epsilon\}}\frac{u(x+z)-u(x)}{|z|^{n+2s}}\, dz\bigg]^p\bigg)^{1/p}=0,
	\end{equation*}
	and
	\begin{equation*}
		\bigg(\int_{\Omega_\epsilon} \, dx\bigg[\int_{\Omega_{\uppercase\expandafter{\romannumeral 2}}\cap\{|z|\leqslant \epsilon\}}\frac{-u(x)}{|z|^{n+2s}}\, dz\bigg]^p\bigg)^{1/p}=0.
	\end{equation*}
	Thus, we get that
	\begin{align*}
	&\bigg(\int_{\Omega}\, dx \bigg[\int_{\Omega_{\uppercase\expandafter{\romannumeral 2}}\cap\{|z|\leqslant \epsilon\}}\frac{u(x+z)-u(x)}{|z|^{n+2s}}\, dz\bigg]^p\bigg)^{1/p}\\
	\leqslant &\  \bigg(\int_{\Omega\backslash\Omega_\epsilon}\, dx\bigg[	\int_{\Omega_{\uppercase\expandafter{\romannumeral 2}}\cap\{|z|\leqslant \epsilon\}}\frac{\| \nabla u(x)\|_{\infty}}{|z|^{n+2s-1-\gamma} d(x)^\gamma} \, dz\bigg]^p\bigg)^{1/p}\\
	\leqslant &\ C(n,p,\gamma,s)\| \nabla u(x)\|_{\infty}\epsilon^{\gamma+1-2s}\bigg(\int_{\Omega\backslash\Omega_\epsilon}\frac{1}{d(x)^{\gamma p}}\, dx\bigg)^{1/p}\\
	\leqslant &\ C(n,p,\gamma,s)\| \nabla u(x)\|_{\infty}\epsilon^{\gamma+1-2s}\bigg(\int_{y\in\partial\Omega}\, dy \int_{0<|x-y|\leqslant \epsilon}\frac{1}{|x-y|^{\gamma p}}\, dx\bigg)^{1/p}\\
	\leqslant &\ C(n,p,\gamma,s,\Omega)\| \nabla u(x)\|_{\infty}\epsilon^{\gamma+1-2s} \epsilon^{-\gamma+\frac{n}{p}}\\
	\leqslant &\ C(n,p,\gamma,s,\Omega)\epsilon^{1-2s+\frac{n}{p}}\|u\|_{W^{2,p}(\Omega)}.
	\end{align*}
		Also,
	\begin{align*}
	&\bigg(\int_{\Omega} \, dx\bigg[\int_{\Omega_{\uppercase\expandafter{\romannumeral 2}}\cap\{|z|\leqslant \epsilon\}}\frac{-u(x)}{|z|^{n+2s}}\, dz\bigg]^p\bigg)^{1/p}\\
	\leqslant\,&\  \bigg(\int_{\Omega\backslash\Omega_\epsilon}\, dx\bigg[\int_{\Omega_{\uppercase\expandafter{\romannumeral 2}}\cap\{|z|\leqslant \epsilon\}}\frac{\|\nabla u(x)\|_{\infty}}{|z|^{n+2s-1-\gamma} d(x)^{\gamma}} \, dz\bigg]^p\bigg)^{1/p}\\
	\leqslant\,& \ C(n,p,\gamma,s,\Omega)\epsilon^{1-2s+\frac{n}{p}}\|u\|_{W^{2,p}(\Omega)}.
	\end{align*}
	Combining the above two estimates, we get
	\begin{align*}
	\|\uppercase\expandafter{\romannumeral 2}_{|z|\leqslant\epsilon}\|_{p}
	&\leqslant C(n,p,s,\Omega)o(\epsilon)\|u\|_{W^{2,p}(\Omega)}.
	\end{align*}
Therefore, as argued in the proof of Lemma~\ref{lemma Lp estimate about 1/2 operator} 	 for $ \uppercase\expandafter{\romannumeral 1},\uppercase\expandafter{\romannumeral 2}_{\epsilon<|z|<1} $, and $ \uppercase\expandafter{\romannumeral 3} \sim \uppercase\expandafter{\romannumeral 5} $, we obtain the desired result.
\end{proof}

 \subsubsection{Towards the existence of the solution to $ -\Delta u+(-\Delta)^su+\lambda u=f $}\label{section 4.12}

Now we develop some preliminary material needed to establish the existence result
in Theorem~\ref{theorem w 2p}. Here, we look at a linear perturbation of the equation in Theorem~\ref{theorem w 2p}.

 \begin{Lemma}\label{lemma existence of lambda  }
 	Under the assumptions of Theorem~\ref{theorem w 2p} and $ s\in(0,\frac{1}{2}) $, let  $ \lambda_1> 0$ be given, large enough and independent of $ f $. Then, the problem
 	\begin{equation}\label{eq: lambda equation}
 	-\Delta u+(-\Delta)^su+\lambda u=f, \qquad \text{in } \Omega
 	\end{equation}
 	has a unique solution $ u\in W^{2,p}(\Omega)\cap W^{1,p}_0(\Omega) $ for any $ \lambda\geqslant \lambda_1$.
 \end{Lemma}
\begin{proof}
	For a given $ w\in W^{2,p}(\Omega)\cap W^{1,p}_0(\Omega) $,  the problem
	\begin{equation*}
		-\Delta u+\lambda u=f-(-\Delta)^sw, \qquad \text{in } \Omega
	\end{equation*}
has a unique solution $ u\in W^{2,p}(\Omega)\cap W^{1,p}_0(\Omega) $ since $f -(-\Delta)^s w\in L^p(\Omega).$ Thus, we can define a mapping $ T_\lambda:w\rightarrow u $ of $ W^{2,p}(\Omega)\cap W^{1,p}_0(\Omega) $ into itself. We are going to prove that it is possible to choose a positive number $ \lambda_1$ large enough in such a way that $ T_\lambda $ is a contraction mapping for any $ \lambda\geqslant \lambda_1$.	
	
	Indeed, let $ w_1,w_2\in W^{2,p}(\Omega)\cap W^{1,p}_0(\Omega) $ and $ T_\lambda w_1=u_1, T_\lambda w_2=u_2 $. Then,
	\begin{equation*}
		-\Delta (u_1-u_2)+\lambda (u_1-u_2)=-(-\Delta)^s(w_1-w_2), \qquad \text{in } \Omega.
	\end{equation*}
	Due to \eqref{eq:lapalce estimate 1},  Lemma~\ref{lemma Lp estimate about fractional operator} and the Gagliardo-Nirenberg interpolation inequality, we have that
	\begin{align*}
		\|u_1-u_2\|_{W^{2,p}(\Omega)}&\leqslant C\|-(-\Delta)^s(w_1-w_2)\|_{p}\\
		&\leqslant C\bigg[\epsilon\|w_1-w_2\|_{W^{2,p}(\Omega)}+\tau(\epsilon)\|w_1-w_2\|_{p}\bigg].
	\end{align*}
	Using \eqref{eq:lapalce estimate 2} instead of \eqref{eq:lapalce estimate 1}, we can see that
	\begin{equation*}
		(\lambda-\lambda_0)\|u_1-u_2\|_{p}\leqslant C\bigg[\epsilon\|w_1-w_2\|_{W^{2,p}(\Omega)}+\tau(\epsilon)\|w_1-w_2\|_{p}\bigg].
	\end{equation*}
Since $ C $ is independent of $ \lambda $, then we can provide $  W^{2,p}(\Omega)\cap W^{1,p}_0(\Omega) $ with the equivalent norm
	\begin{equation*}
		\||\cdot \||_{W^{2,p}}:= \|\cdot \|_{W^{2,p}(\Omega)}+(\lambda-\lambda_0)\|\cdot\|_{p}.
	\end{equation*}
	First, let~$ k\in(0,1) $ be given. We choose $ \epsilon>0$ small enough such that
	\begin{equation*}
		C\epsilon\leqslant\frac{k}{2},
	\end{equation*}
	 Next, take $ \lambda_1>0$ large enough such that
	 \begin{equation*}
	 	C\tau(\epsilon)<(\lambda_1-\lambda_0)\frac{k}{2}.
	 \end{equation*}
	 From this, for $ \lambda\geqslant\lambda_1 $ it follows  that
	 \begin{equation*}
	 \||T_\lambda w_1-T_\lambda w_2 \||_{W^{2,p}}\leqslant k \||w_1-w_2 \||_{W^{2,p}},
	 \end{equation*}
	i.e., $T_\lambda$ is a contraction, hence the result holds as desired.
\end{proof}

  Next, we consider the case~$ s=\frac{1}{2}$.

  \begin{Lemma}\label{lemma existence of lambda 1/2  }
  	Under the assumptions of Theorem~\ref{theorem w 2p}, let  $ \lambda'_1>0$ be given, large enough and independent of $f$. Then the problem
  	\begin{equation}\label{eq: lambda equation 1/2}
  	-\Delta u+(-\Delta)^{\frac{1}{2}}u+\lambda u=f, \qquad \text{in } \Omega
  	\end{equation}
  	has a unique solution $ u\in W^{2,p}(\Omega)\cap W^{1,p}_0(\Omega) $ for any $ \lambda\geqslant \lambda'_1$.
  \end{Lemma}
  \begin{proof} The proof is close in spirit to
  that of Lemma~\ref{lemma existence of lambda  }, relying here on Lemma~\ref{lemma Lp estimate about 1/2 operator}
instead of Lemma~\ref{lemma Lp estimate about fractional operator}. We provide full details for the reader's convenience.

  	For a  $ w\in W^{2,p}(\Omega)\cap W^{1,p}_0(\Omega) $,  the problem
  	\begin{equation*}
  	-\Delta u+\lambda u=f-(-\Delta)^{\frac{1}{2}}w, \qquad \text{in } \Omega
  	\end{equation*}
  	has a unique solution $ u\in W^{2,p}(\Omega)\cap W^{1,p}_0(\Omega) $ since $f -(-\Delta)^{\frac{1}{2}}w\in L^p(\Omega).$ Thus, we can define a mapping $ T_\lambda:w\rightarrow u $ of $ W^{2,p}(\Omega)\cap W^{1,p}_0(\Omega) $ into itself. We are going to prove that it is possible to choose a positive number $ \lambda'_1$ large enough such that $T_\lambda $ is a contraction mapping for any $ \lambda\geqslant \lambda'_1$.	
  	
  	Indeed, let $ w_1,w_2\in W^{2,p}(\Omega)\cap W^{1,p}_0(\Omega) $ and $ T_\lambda w_1=u_1, T_\lambda w_2=u_2$. Then,
  	\begin{equation*}
  	-\Delta (u_1-u_2)+\lambda (u_1-u_2)=-(-\Delta)^{\frac{1}{2}}(w_1-w_2), \qquad \text{in } \Omega.
  	\end{equation*}
  	Due to \eqref{eq:lapalce estimate 1},  Lemma~\ref{lemma Lp estimate about 1/2 operator} and the Gagliardo-Nirenberg interpolation inequality, we have that
  	\begin{align*}
  	\|u_1-u_2\|_{W^{2,p}(\Omega)}&\leqslant C\|-(-\Delta)^s(w_1-w_2)\|_{p}\\
  	&\leqslant C\bigg[({o(\epsilon)}+\tau(\epsilon)\delta)\|w_1-w_2\|_{W^{2,p}(\Omega)}+
  	\frac{\tau(\epsilon)}{\delta}\|w_1-w_2\|_{p}\bigg].
  	\end{align*}
  	Using \eqref{eq:lapalce estimate 2} instead of \eqref{eq:lapalce estimate 1}, one can get that
  	\begin{equation*}
  	(\lambda-\lambda_0)\|u_1-u_2\|_{p}\leqslant C\bigg[({o(\epsilon)}+\tau(\epsilon)\delta)\|w_1-w_2\|_{W^{2,p}(\Omega)}+
  	\frac{\tau(\epsilon)}{\delta}\|w_1-w_2\|_{p}\bigg].
  	\end{equation*}
  	We observe that $C$ is independent of $ \lambda$, hence we can provide $  W^{2,p}(\Omega)\cap W^{1,p}_0(\Omega) $ with the equivalent norm
  	\begin{equation*}
  	\||\cdot \||_{W^{2,p}}= \|\cdot \|_{W^{2,p}(\Omega)}+(\lambda-\lambda_0)\|\cdot\|_{p}.
  	\end{equation*}
  	First, let $ k\in(0,1) $. We can choose $ \epsilon, \delta> 0$ small enough such that
  	\begin{equation*}
  	C[{o(\epsilon)}+\tau(\epsilon)\delta]\leqslant\frac{k}{2},
  	\end{equation*}
  	Next, take $ \lambda'_1>0$ large enough such that
  	\begin{equation*}
  	C\frac{\tau(\epsilon)}{\delta}<(\lambda_1-\lambda_0)\frac{k}{2}.
  	\end{equation*}
  	From this, for $ \lambda\geqslant\lambda'_1 $ it follows  that
  	\begin{equation*}
  	\||T_\lambda w_1-T_\lambda w_2 \||_{W^{2,p}}\leqslant k \||w_1-w_2 \||_{W^{2,p}},
  	\end{equation*}
  	i.e., $ T_\lambda $ is a contraction, hence the result holds as desired.
  \end{proof}

\subsubsection{Proof of Theorem~\ref{theorem w 2p}}\label{subsection 4.13}

Our objective is now to complete the proof of the existence result stated in Theorem~\ref{theorem w 2p}.
For this, to begin with, we give some auxiliary results related to the Maximum Principle.
\begin{Lemma}\label{lemma infty bounded}
	Let   $ \Omega\subset  \mathbb{R}^n $ be $ C^{1,1} $ domain  and $ s\in (0,1)$.  Supposing that $ h\in L^\infty(\Omega)$ and $ \lambda\geqslant\max\{\lambda_1,\lambda'_1\}$, where the positive numbers $\lambda_1$ and $\lambda'_1$ are from Lemmata~\ref{lemma existence of lambda  } and \ref{lemma existence of lambda 1/2  } respectively. Then, the problem
	\begin{equation}\label{eq:infty h}
		-\Delta u+(-\Delta)^su+\lambda u=h, \qquad \text{in } \Omega
	\end{equation}
	has a unique solution $ u\in W^{2,p}(\Omega)\cap W^{1,p}_0(\Omega) $ and
	\begin{equation*}
		\|u\|_{\infty}\leqslant c(\lambda)\  \|h\|_{\infty},
	\end{equation*}
	where $ c(\lambda)<\frac{1}{\lambda}$.
\end{Lemma}

Before that, we point out that
by combining Lemma~\ref{lemma existence of lambda  } with the Sobolev embedding theorem,  we infer that $ u\in C^0(\overline{\Omega})\subset L^\infty(\Omega)$. In order to complete the proof of Lemma~\ref{lemma infty bounded}, it remains to prove the next two Lemmata~\ref{lemma existence of w} and~\ref{lemma: infty norm <lambda}.

\begin{Lemma}\label{lemma existence of w}
	 There exists $ w\in C^\infty(\overline{\Omega})\cap W^{1,\infty}(\mathbb{R}^n) $  such that for every $ s\in(0,1) $
	\begin{equation*}
	-\Delta w+(-\Delta)^sw\geqslant 1,\quad \text{ in }\Omega,\quad w>0 \text{ in } \overline{\Omega}, w\geqslant0 \text{ on } \mathbb{R}^n.
	\end{equation*}
\end{Lemma}
\begin{proof}
Since $ \Omega $ is bounded in $ \mathbb{R}^n $, there exist $x_0\in \mathbb{R}^n$ and a positive constant $R$ such that
$ \overline{\Omega}\subset B_{\frac{3}{4}R}(x_0)\backslash B_{\frac{1}{4}R}(x_0)$ where $B_R(x_0):=\{x\in\mathbb{R}^n : |x-x_0|<R\} $. Without loss of generality, one could assume  $x_0=0$ by the translation-invariance of the mixed operator $\mathcal{L}$ and denote $B_R= B_R(0)$.

Let
 \begin{equation}\label{eq:function LW > 1}
 w(x):=\begin{cases}
 1-e^{\beta(|x|^2-R^2)},  \qquad &|x|\leqslant R,\\
 0&|x|>R,
 		\end{cases}
 \end{equation}
 where $ \beta>0 $ will be determined below.

One can calculate that
 \begin{equation*}
 	-\Delta w(x)=e^{\beta(|x|^2-R^2)}(2n\beta+4\beta^2|x|^2).
 \end{equation*}
Also, $ w $ is concave on $ \overline{B_R} $, and for every $ x\in\Omega$, if $ |z|\leqslant 1$, then $ x\pm z\in B_R$.

Therefore,
 \begin{align*}
 	(-\Delta)^sw(x)&=-\frac{1}{2}\int_{\{|z|\leqslant 1\}\cup\{ |z|> 1\}}\frac{w(x+z)+w(x-z)-2w(x)}{|z|^{n+2s}}\, dz\\
 	&\geqslant-\frac{1}{2}\int_{\{|z|> 1\}}\frac{w(x+z)+w(x-z)-2w(x)}{|z|^{n+2s}}\, dz\\
 	&=-\int_{\{|z|> 1\}}\frac{w(x+z)-w(x)}{|z|^{n+2s}}\, dz\\
 	&=-\int_{\{|z|> 1\}\cap\{|x+z|< |x|\}} \frac{w(x+z)-w(x)}{|z|^{n+2s}}\, dz-\int_{\{|z|> 1\}\cap\{|x+z|\geqslant |x|\}} \frac{w(x+z)-w(x)}{|z|^{n+2s}}\, dz\\
 	&\geqslant-\int_{\{|z|> 1\}\cap\{|x+z|< |x|\}} \frac{w(x+z)-w(x)}{|z|^{n+2s}}\, dz\\
 	&\geqslant -C(s) e^{\beta(|x|^2-R^2)}.
 \end{align*}
 As a consequence of this, we get
 \begin{align*}
 	-\Delta w+(-\Delta)^sw&\geqslant e^{\beta(|x|^2-R^2)} \left(2n\beta+4\beta^2|x|^2 - C(s) \right)\\
 	&\geqslant e^{-\frac{15\beta R^2}{16}} \left(2n\beta+\frac{\beta^2R^2}{4}-C(s) \right)
 	\geqslant \alpha_0 ,
 \end{align*}
 for some $ \alpha_0>0 $ if $ \beta$ is large enough, which can immediately imply Lemma~\ref{lemma existence of w}.
\end{proof}

\begin{Lemma}\label{lemma: infty norm <lambda}
	The problem
	\begin{equation*}
	-\Delta u_\lambda+(-\Delta)^su_\lambda+\lambda u_\lambda=1, \qquad \text{in } \Omega
	\end{equation*}
	has a unique solution $ w_\lambda\in W^{2,p}(\Omega)\cap L^\infty(\Omega) $
	 and $ \|w_\lambda\|_{\infty}<\frac{1}{\lambda} $.
\end{Lemma}
\begin{proof}
The existence and uniqueness of~$w_\lambda$ are guaranteed by Lemmata~\ref{lemma existence of lambda  } and \ref{lemma existence of lambda 1/2  }.

Then, the Sobolev embedding theorem implies  that  $ w_\lambda\in C^0(\overline{\Omega})\subset L^\infty(\Omega) $. Also, by  Lemma~\ref{lemma existence of w}, it follows that there exists
 $ w\in C^2(\overline{\Omega})\cap W^{1,\infty}(\mathbb{R}^n) $  such that
 \begin{equation*}
 -\Delta w+(-\Delta)^sw\geqslant 1,\quad \text{ on } \Omega, \quad w>0 \text{ in } \overline{\Omega}, \quad w\geqslant0 \text{ on } \mathbb{R}^n.
 \end{equation*}
We set $ \varphi(w)=\frac{1}{\lambda}(1-e^{-\lambda w}) $, and note that $ \varphi $ is concave for all $ w\geqslant 0 $. Thus, for $ x,y\in \mathbb{R}^n $ ,
\begin{equation*}
	\varphi(w(x))-\varphi(w(y))\geqslant \varphi'(w(x))(w(x)-w(y)).
\end{equation*}
 Computing $ (-\Delta)^s\varphi(w(x)) $ for $ x\in \Omega $
 \begin{align*}
 	(-\Delta)^s\varphi(w(x))&\geqslant\int_{\mathbb{R}^n}\frac{\varphi'(w(x))(w(x)-w(y))}{|x-y|^{n+2s}}\, dy\\
 	&=e^{-\lambda w}(-\Delta)^sw(x).
 \end{align*}
 As a consequence,
 \begin{equation}
 \begin{split}
 	&-\Delta\varphi(w(x))+(-\Delta)^s\varphi(w(x))+\lambda\varphi(w(x))\\
	\geqslant& e^{-\lambda w} \left(-\Delta w+\lambda \sum_{i=1}^{n}(D_iw)^2+(-\Delta)^s w\right)+1-e^{-\lambda w}
 	\geqslant 1.\label{eq: >1}
	\end{split}
 \end{equation}
 Combining with \eqref{eq: >1}, we get
 \begin{equation*}
 		-\Delta (\varphi(w)-w_\lambda)+(-\Delta)^s(\varphi(w)-w_\lambda)+\lambda (\varphi(w)-w_\lambda)\geqslant 0, \qquad \text{in } \Omega.
 \end{equation*}
 The fact $ \varphi(w),w_\lambda\in C^0(\overline{\Omega})\cap W^{2,p}(\Omega) $ with $ p>n $ implies that $ \varphi(w),w_\lambda\in H^1(\mathbb{R}^n) $.

 Applying the { Maximum Principle (see \cite[Theorem 1.2]{MR4387204})}, we obtain
 \begin{equation*}
 \varphi(w(x))\geqslant w_\lambda(x), \quad \text{ a.e. on  } \Omega,
 \end{equation*}
 since $ \varphi(w(x))- w_\lambda(x)\geqslant 0 $ a.e. in $ \mathbb{R}^n\backslash\Omega $.

 Hence, due to the fact that $ \max\limits_{\overline\Omega}\varphi(w(x))<\frac{1}{\lambda} $,
we conclude that $ \|w_\lambda\|_{\infty}<\frac{1}{\lambda} $ on $ {\Omega} $, as desired.
\end{proof}

 \begin{proof}[Proof of  Lemma~\ref{lemma infty bounded}] \label{lemma infty bounded PAGE}
 Recalling \eqref{eq:infty h}, we get
 \begin{equation*}
 	-\Delta \left(\frac{u}{\|h\|_{\infty}}\right)+(-\Delta)^s\left(\frac{u}{\|h\|_{\infty}}\right)+\lambda \left(\frac{u}{\|h\|_{\infty}}\right)\leqslant 1, \qquad \text{in } \Omega,
 \end{equation*}
 	combining with the Lemma~\ref{lemma: infty norm <lambda}, we see that
 	\begin{equation*}
 			-\Delta \left(w_\lambda-\frac{u}{\|h\|_{\infty}}\right)+(-\Delta)^s\left(w_\lambda-\frac{u}{\|h\|_{\infty}}\right)+\lambda \left(w_\lambda-\frac{u}{\|h\|_{\infty}}\right)\geqslant 0, \qquad \text{in } \Omega.
 	\end{equation*}
 	 The { Maximum Principle} implies that
 	 \begin{equation*}
 	 	\|u\|_{\infty}\leqslant w_\lambda\  \|h\|_{\infty}.
 	 \end{equation*}
 	The  desired result thus follows by  Lemma~\ref{lemma: infty norm <lambda}.
\end{proof}

We are finally ready to prove Theorem~\ref{theorem w 2p}.

\begin{proof}[Proof of Theorem~\ref{theorem w 2p}]
	Fixed $ \lambda\geqslant\max\{\lambda_1,\lambda'_1\}$ such that we can solve the equation \eqref{eq: lambda equation} or \eqref{eq: lambda equation 1/2} with $ f\in L^p(\Omega)$. We consider the sequence of functions defined in the following way.
	\begin{equation*}
			-\Delta u_0+(-\Delta)^su_0+\lambda u_0=f,\quad u_0\in  W^{2,p}(\Omega)\cap W^{1,p}_0(\Omega) ,
	\end{equation*}
	and  as $ u_k $ is defined in $ W^{2,p}(\Omega)\cap W^{1,p}_0(\Omega) $, $ u_{k+1} $ is the solution of
	\begin{equation}\label{eq:iteration main equation}
		-\Delta u_{k+1}+(-\Delta)^su_{k+1}+\lambda u_{k+1}=f+\lambda u_k,\quad u_{k+1}\in  W^{2,p}(\Omega)\cap W^{1,p}_0(\Omega).
	\end{equation}
	We can consider that $ u_{k+1}-u_{k}  $ is the solution of
	\begin{equation}\label{eq:interation equation}
		-\Delta (u_{k+1}-u_{k})+(-\Delta)^s(u_{k+1}-u_{k})+\lambda (u_{k+1}-u_{k})=\lambda(u_{k}-u_{k-1}).
	\end{equation}
	Now, let
	\begin{equation*}
		j=\left[\frac{n}{2p}\right] ,\quad \frac{1}{p_i}=\frac{1}{p}-\frac{2i}{n}.
	\end{equation*}
	Assuming that $ u_{-1}=0, $ $ u_{k+1}-u_{k} $ is the solution of equation \eqref{eq: lambda equation} or \eqref{eq: lambda equation 1/2} with $ f=\lambda(u_{k}-u_{k-1}) $. By the preceding we can deduce other regularity properties of $ u_{k+1}-u_{k} $.
	
	Indeed, since $ u_0\in W^{2,p}(\Omega) $, we also have $ u_0\in L^{p_1}(\Omega). $ with $ \frac{1}{p_1}=\frac{1}{p}-\frac{2}{n} $. Hence,
	\begin{equation*}
		u_1-u_0\in W^{2,p}(\Omega)\cap W^{2,p_1}(\Omega), \quad \frac{1}{p_1}=\frac{1}{p}-\frac{2}{n}.
	\end{equation*}
	Then $ u_1-u_0\in L^{p_2}(\Omega). $ Hence
	\begin{equation*}
	u_2-u_1\in W^{2,p}(\Omega)\cap W^{2,p_2}(\Omega), \quad \frac{1}{p_2}=\frac{1}{p}-\frac{4}{n}.
	\end{equation*}
	
	By induction we prove
	\begin{equation*}
		u_{k}-u_{k-1}\in W^{2,p}(\Omega)\cap W^{2,p_k}(\Omega),\quad \frac{1}{p_k}=\frac{1}{p}-\frac{2k}{n},
	\end{equation*}
	as long as $ k<\frac{n}{2p}$.
	\begin{itemize}
		\item[Case 1.] If $ j=\left[\frac{n}{2p}\right]=0 $, then $ 0<\frac{n}{2p}<1 $. Hence $ u_0\in C^0(\overline\Omega)\subset L^\infty(\Omega). $
		\item[Case 2.] If $ j<\frac{n}{2p}<j+1 $, then $ u_{j}-u_{j-1}\in  W^{2,p}(\Omega)\cap W^{2,p_j}(\Omega)  $ with $ p_j>\frac{n}{2} $. Hence $ u_{j}-u_{j-1}\in C^0(\overline\Omega)\subset L^\infty(\Omega) $.
		\item[Case 3.] If $ j=\frac{n}{2p}\geqslant 1 $, then $ u_{j-1}-u_{j-2}\in W^{2,p}(\Omega)\cap W^{2,\frac{n}{2}}(\Omega)\subset W^{1,\frac{n}{2}}(\Omega)  $. Hence $ u_{j-1}-u_{j-2}\in L^n(\Omega)$. Therefore, $ u_{j}-u_{j-1}\in W^{2,p}(\Omega)\cap W^{2,n}(\Omega) $. Due to the interpolation inequality, we know $ u_{j}-u_{j-1}\in W^{2,\theta p+(1-\theta)n}(\Omega) $ with $ \theta\in [0,1].$ Let
		\begin{equation*}
			q:=\theta p+(1-\theta)n=n(1-\theta(1-\frac{1}{2j})).
		\end{equation*}
		In particular, we observe that
		\begin{equation*}
			2-\frac{n}{q}=2-\frac{1}{1-\theta(1-\frac{1}{2j})}=\frac{1-2\theta(1-\frac{1}{2j})}{1-\theta(1-\frac{1}{2j})}=\alpha,
		\end{equation*}
		and we can choose $ \theta>0$ small enough such that $ \alpha\in(0,1) $. From this, it follows that $ u_{j}-u_{j-1}\in C^{0,\alpha}(\overline\Omega) \subset L^\infty(\Omega). $
	\end{itemize}
	Thus in all cases, Lemma~\ref{lemma infty bounded} and \eqref{eq:interation equation} can imply that, for all~$ k\geqslant j$, $u_{k}-u_{k-1}\in L^\infty(\Omega).  $ Moreover,
	\begin{equation*}
		\|u_{k+1}-u_{k}\|_{\infty}\leqslant K \|u_{k}-u_{k-1}\|_{\infty},
	\end{equation*}
	where $ 0<K<1 $. It is then immediate to see that $ \{u_k\} $ is a Cauchy sequence in $ L^\infty(\Omega) $ and is bounded. Recalling \eqref{eq:iteration main equation},  utilizing \eqref{eq:lapalce estimate 1}  and   \eqref{eq:lapalce estimate 2}, we see that
\begin{align*}
 &\|u_{k+1}\|_{W^{2,p}(\Omega)}+(\lambda-\lambda_0)\|u_{k+1}\|_{p}\\
 \leqslant &\, 2C\bigg(\epsilon\|u_{k+1}\|_{W^{2,p}(\Omega)}+\tau(\epsilon)\|u_{k+1}\|_{p}+\|f\|_{p}+\lambda \|u_k\|_{p}\bigg)\\
 \leqslant &\,k\bigg(\|u_{k+1}\|_{W^{2,p}(\Omega)}+(\lambda-\lambda_0)\|u_{k+1}\|_{p}\bigg)+C \|f\|_{p}+C\lambda \|u_k\|_{\infty},
\end{align*}
 where the constant $ C $ is independent of $ \lambda $ and $u_k$. It follows that $ \{u_k \} $ is bounded in $ W^{2,p}(\Omega)\cap W^{1,p}_0(\Omega)$. Thus there exists a subsequence, which we relabel as $ \{u_k\} $, converging weakly to a function $ u\in W^{2,p}(\Omega)\cap W^{1,p}_0(\Omega)$. Since
\begin{equation*}
	\int_{\Omega} gD^\alpha u_k\rightarrow \int_{\Omega}gD^\alpha u
\end{equation*}
for all $ |\alpha|\leqslant 2 $ and $ g\in L^{p/(p-1)}(\Omega) $, we must have
\begin{equation*}
	\int_{\Omega} g(-\Delta u_k)\rightarrow \int_{\Omega}g (-\Delta u).
\end{equation*}
Similarly,
\begin{equation*}
	\int_{\Omega} g(-\Delta )^su_k\rightarrow \int_{\Omega}g (-\Delta )^su
\end{equation*}
since $ (-\Delta)^s\in \mathscr{L}(W^{2,p}(\Omega)\cap W_0^{1,p}(\Omega),L^p(\Omega)) $.
Combining the above results and \eqref{eq:iteration main equation},  we obtain
\begin{equation*}
	-\Delta u+(-\Delta)^su=f,\quad u\in  W^{2,p}(\Omega)\cap W^{1,p}_0(\Omega).
\end{equation*}
In particular, we can set (i) $ \epsilon\in\left(0,\frac{1}{2C}\right) $ in \eqref{eq: Lp estimate of fractional s} for $ s\in(0,\frac{1}{2}) $ and (ii) $\epsilon>0$ such $ {o(\epsilon)}<\frac{1}{2C} $ in \eqref{eq: Lp estimate of 1/2} for $ s=\frac{1}{2} $, one can obtain
\begin{align*}
	\|u\|_{W^{2,p}(\Omega)}\leqslant C_1\bigg(\|u\|_{p}+\|f\|_{p}\bigg),
\end{align*}
 where the constant $C_1$ is  independent of $ u$. The result holds as desired.
\end{proof}

 \subsection{Proof of $ C^{1,\alpha} $-regularity}

 Now we complete the proof of Theorem~\ref{theorem C^1 alpha}.

 \begin{proof}[Proof of Theorem~\ref{theorem C^1 alpha}]
 Let $ u\in X_0^1 $ be the weak solution of \eqref{eq:g main equation}. Theorem~\ref{th: regularity} implies that $ g(x,u)\in L^\infty(\Omega)$.

 Moreover, due to Theorem~\ref{theorem w 2p}, we  deduce that  there exists a unique solution $ v\in W^{2,p}(\Omega)\cap W^{1,p}_0(\Omega) $ for every  $ p>n $. By combining this with the Sobolev embedding inequality, we obtain $ v\in C^{1,\alpha}(\overline\Omega)\subset X_0^1 $ for any $ \alpha\in(0,1)$.

Finally, the Lax-Milgram Theorem yields that $u=v$. Hence, the proof of Theorem~\ref{theorem C^1 alpha} is completed.
  \end{proof}

\section*{Acknowledgments}
The authors would like to thank the anonymous referee for carefully reading the manuscript and the valuable comments  and  suggestions on it.

\bibliographystyle{alpha}

\bibliography{reference}

\end{document}